\documentclass[11pt]{amsart}
\usepackage[utf8]{inputenc}
\usepackage{amsmath,amssymb}
\usepackage{wrapfig}
\usepackage{url}
\usepackage{mathtools}
\usepackage{graphicx}
\usepackage{stmaryrd}
\usepackage{amsthm}
\usepackage{xcolor}
\usepackage[colorlinks=true,linkcolor=blue,citecolor=blue]{hyperref}
\usepackage[shortlabels]{enumitem}
\usepackage{comment}
\usepackage{relsize}
\usepackage{dsfont}
\usepackage{stmaryrd}
\usepackage{geometry}
\usepackage{setspace}
\usepackage{mathrsfs} 
\usepackage{orcidlink}
\usepackage{slashed}
\usepackage{yhmath}
\makeatletter

\newtheorem*{theorem*}{Theorem}
\newtheorem{prop*}{Proposition}

\newtheorem{theorem}{Theorem}
\numberwithin{theorem}{section} 
\numberwithin{equation}{section}
\newtheorem{corollary}[theorem]{Corollary}

\newtheorem{prop}[theorem]{Proposition} 
\theoremstyle{definition}

\numberwithin{example}{section}

\newtheorem{lemma}[theorem]{Lemma}
\theoremstyle{definition}
\newtheorem{obs}{Remark}
\theoremstyle{definition}
\newtheorem{definition}[theorem]{Definition}

\newcommand{\R}{\mathbb{R}}
\newcommand{\Z}{\mathbb{Z}}
\newcommand{\N}{\mathbb{N}}
\newcommand{\T}{\mathbb{T}}
\newcommand{\C}{\mathbb{C}}
\newcommand{\Tr}{\operatorname{Tr}}
\newcommand{\HS}{\operatorname{HS}}
\renewcommand{\S}{\mathbb{S}^3}
\title[Directional Poincaré inequality on compact Lie groups]{Directional Poincaré inequality on compact Lie groups}
\newcommand{\defeq}{\vcentcolon=}

\author[P. L. Dattori da Silva]{Paulo L. Dattori da Silva} 
\address{
  Departamento de Matemática,
 Instituto de Ciências Matemáticas e de Computação,
 Universidade de São Paulo, São Paulo, Brazil
  }
\email{dattori@icmc.usp.br}

\author[A. Kowacs]{André Pedroso Kowacs} 
\address{
  Departamento de Matemática,
 Instituto de Ciências Matemáticas e de Computação,
 Universidade de São Paulo, São Paulo, Brazil
  }
\email{andrekowacs@gmail.com}



\date{\today}

\begin{document}

\begin{abstract}
    We extend the directional Poincaré inequality on the torus, introduced by Steinerberger  in [Ark. Mat. 54 (2016), pp. 555--569], to the setting of compact Lie groups. We provide necessary and sufficient conditions for the existence of such an inequality based on estimates on the eigenvalues of the global symbol of the corresponding vector field. 
    We also prove that such refinement of the Poincaré  inequality holds for a left-invariant vector field on a compact Lie group $G$ if and only if the vector field is globally solvable, and extend this equivalence to tube-type vector fields on $\T^1\times G$.
\end{abstract}
\maketitle

\section{Introduction}

In \cite{directional_Poincare}, Steinerberger introduced the notion of a refinement for the classical Poincaré inequality on the torus, called the {\it directional Poincaré inequality}, and proved that it holds for an uncountable but Lebesgue-null set of directions $\alpha$. More precisely, he showed that there is a set $\mathcal{B}\subset\mathbb{R}^n$ of directions such that for every $\alpha\in\mathcal{B}$ there is $c_\alpha>0$ so that
\begin{equation}\label{DPI_intro}
	\|\nabla f\|_{L^2(\mathbb{T}^n)}^{n-1}\| \langle \nabla f,\alpha\rangle \|_{L^2(\mathbb{T}^n)}\geq c_\alpha\|f\|_{L^2(\mathbb{T}^n)}^n, 
\end{equation}
for every $f\in H^1(\mathbb{T}^n)$ with mean-value zero, and if $n\geq2$ then $\mathcal{B}$ is uncountable but Lebesgue-null. Here, and for the rest of this paper, $\T^n=\R^n/(2\pi\Z)^n$ denotes the $n$-dimensional torus.

As the torus is a compact Lie group, and the directional derivative $\langle \nabla ,\alpha\rangle=\sum \alpha_j\partial_{x_j}$ is a (left-)invariant vector field under its own group action, a natural question to ask is if an analogue of such an inequality can also hold for left-invariant vector fields on other compact Lie groups.
In this paper we show that in the same formulation, the answer to this question is negative. In fact, we prove that (see Section \ref{sect_nec_suff}, Corollary \ref{coro_poincare_implies_torus}):

\emph{  Let $Y=\langle \nabla_{G},\alpha\rangle$ be a left-invariant real vector field on a compact Lie group $G$ and $\delta\geq 1$. If there exists $c>0$ such that 
    \begin{equation*}
        \|\nabla_{G} f\|_{L^2(G)}^{\delta-1}\| Yf\|_{L^2(G)}\geq c\|f\|_{L^2(G)}^\delta,
    \end{equation*}
    for every $f\in H^1(G)$ with mean-value zero, then $G$ is a torus.
}

 In summary, this follows from the fact that under the conditions stated above, 
 the vector field $Y$ must be globally hypoelliptic, and 
 and so a result from Greenfield and Wallach in  \cite{GW_GH_vector_fields} 
 allows us to conclude that $G$ is a torus.

Hence, in order to obtain a similar  directional Poincaré inequality for invariant vector fields on compact Lie groups, we have to look for another subspace of $H^1(G)$ where the inequality may hold. One can easily verify that no similar inequality can hold in a subspace which intersects the kernel of $Y:H^1(G)\to L^2(G)$, hence the largest subspace where this inequality might hold is $(\ker Y)^\perp\subset H^1(G)$. We prove that a similar inequality indeed can exist as long as the symbol of the vector field $Y$ satisfies certain conditions. Our main result reads as follows (see Section \ref{sect_nec_suff}, Theorem \ref{theo_directional_poincare_general}).

\emph{
    Let $Y=\langle \nabla_{G},\alpha\rangle$ be a left-invariant real vector field on a compact Lie group $G$ and $\delta\geq 1$. There exists $c>0$ such that 
\begin{equation*}
        \|\nabla_{G} f\|_{L^2(G)}^{\delta-1}\| Yf\|_{L^2(G)}\geq c\|f\|_{L^2(G)}^\delta,
    \end{equation*}
    for every $f\in (\ker Y|_{H^1(G)})^\perp$, if and only if there exists $C>0$ such that
\begin{equation*}
        \lambda_{\min}^{>0}[\sigma_Y(\xi)]\geq C\langle \xi \rangle^{-(\delta-1)},
    \end{equation*}
    for all $[\xi]\in\widehat{G}$ such that $\sigma_Y(\xi)\neq 0$, where $\lambda_{\min}^{>0}[\sigma_{Y}(\xi)]$ denotes the smallest non-zero  singular value of $\sigma_Y(\xi)$.
}

As an application, in Section \ref{sect_solvability} we provide characterizations of the global solvability of left-invariant continuous linear operators (Fourier multipliers) on compact Lie groups in terms of the existence of the associated directional Poincaré inequality and, equivalently, on lower bounds for the decay of the smallest non-zero singular values of its symbol.

Finally, inspired  by the connection to global solvability mentioned above,  in Section \ref{sect_tube} we prove that the existence of a directional Poincaré inequality for tube-type vector fields on a product $\T^1\times G$ is also equivalent to the global solvability of the corresponding vector field. In particular, we extend the directional Poincaré inequality on $\T^2$ to all real vector fields on $\T^2$ with a non-vanishing coefficient, in the following result (see Section \ref{sect_tube}, Corollary \ref{coro_tube_t2}):

\emph{
     Let $a\in C^\infty(\T^1)$ be real-valued and consider the vector field on $\T^2$ given by
     \begin{equation*}
         Y=\partial_t+a(t)\partial_x.
     \end{equation*}
     There exist $\delta\geq 2$ and  $c_a>0$ such that
          \begin{equation*}
        \|\nabla_{\T^2} f\|_{L^2(\T^2)}^{\delta-1}\| Y f \|_{L^2(\T^2)}\geq c_a\|f\|_{L^2(\T^2)}^\delta,
    \end{equation*}
    for every $f\in H^1(\T^2)$ with mean-value zero, if and only if 
     $a_0=\frac{1}{2\pi}\int_0^{2\pi} a(t)dt$ is an irrational non-Liouville number. Moreover: in this case we can take $\delta$ greater than the irrationality measure of $a_0$, or equal to $2$ if $a_0$ is algebraic of degree $2$. 
}

We believe these results indicate that this equivalence between the directional Poincaré inequality and global solvability might also hold for more general classes of vector fields on the torus and on other manifolds and, if true, could be an useful tool for the study of global solvability in the future.

\section{Preliminaries}\label{sect_pre}

Let $G$ be a (connected) compact Lie group, and denote by $\widehat{G}$ its unitary dual, that is, the set comprised of equivalence classes of continuous irreducible unitary representations of $G$. From the fact that $G$ is compact, any $\xi\in [\xi]\in\widehat{G}$ is finite dimensional;  we denote its dimension by $d_\xi$ and assume $\xi$ to be matrix-valued. From the Peter-Weyl Theorem, the collection of coefficient functions of representatives of elements in $\widehat{G}$ form an orthogonal Schauder basis for $L^2(G)$, with respect to the (normalized) Haar measure on $G$. Hence, by considering the matrix-valued Fourier coefficients for $f\in L^2(G)$ as defined in \cite{RT2010_book}, and given by  
\begin{equation*}
    \widehat{f}(\xi)\vcentcolon= \int_G f(x)\xi(x)^*dx\in\mathbb{C}^{d_\xi\times d_\xi}, 
\end{equation*}
for every $ [\xi]\in\widehat{G}$, we have the Fourier inversion formula:
\begin{equation*}
    f(x)=\sum_{[\xi]\in\widehat{G}} d_\xi\Tr\left(\xi(x)\widehat{f}(\xi)\right),
\end{equation*}
with convergence for almost every $x$ in $G$, as well as in $L^2(G)$, where in the sum above we choose exactly one representative of each class in $\widehat{G}$, and we maintain this convention throughout the paper. This also yields the analogue of Plancherel's Theorem and Parseval's Formula:
\begin{equation}\label{eq_Plancherel}
    \|f\|_{L^2(G)}=\left(\sum_{[\xi]\in \widehat{G}}d_\xi\|\widehat{f}(\xi)\|_{\HS}^2\right)^{\frac{1}{2}},\quad \langle f,g\rangle_{L^2(G)}=\sum_{[\xi]\in\widehat{G}}d_\xi\Tr\Big(\widehat{f}(\xi)\widehat{g}(\xi)^*\Big)
\end{equation}
for any $f,g\in L^2(G)$, where $\|A\|_{\HS}$ denotes the Hilbert-Schmidt norm of a complex matrix $A$, given by $\sqrt{\Tr(A^*A)}=\sqrt{\sum_{i,j}|A_{ij}|^2}$. 

Let $\mathcal{D}'(G)$ denote the set of distributions on $G$, that is, the set of continuous linear functionals on $C^\infty(G)$. Then its matrix-valued  Fourier coefficients are defined by
\begin{equation*}
    \widehat{u}(\xi)\vcentcolon=\langle u,\xi^*\rangle\in\mathbb{C}^{d_\xi\times d_\xi},
\end{equation*}
for every $ [\xi]\in\widehat{G}$, where the action of $u$ should be interpreted  coefficient-wise, and the brackets denote the distribution-function duality.

For any continuous linear operator $P:C^\infty(G)\to C^\infty(G)$, its matrix-valued global symbol is defined by
\begin{equation*}
    \sigma_P(x,\xi)\vcentcolon= \xi(x)^*P\xi(x),
\end{equation*}
for every $(x,[\xi])\in G\times \widehat{G}$. This symbols yields the quantisation formula:
\begin{equation}\label{eq_quant_formula}
    Pf(x)=\sum_{[\xi]\in\widehat{G}}d_\xi \Tr\left(\xi(x)\sigma_P(x,\xi)\widehat{f}(\xi)\right),
\end{equation}
for every $x\in G$, $f\in C^\infty(G)$. 

If $P:C^\infty(G)\to C^\infty(G)$ is a continuous linear operator which also commutes with left-translations by the group action on itself (left-invariant), then its symbol $\sigma_P(x,\xi)$ does not depend on the variable $x$ and we denote it by  $\sigma_P(\xi)$. In this case, the quantisation formula \eqref{eq_quant_formula}, along with the Fourier inversion formula implies that $\widehat{Pf}(\xi)=\sigma_P(\xi)\widehat{f}(\xi)$, and thus we call it a Fourier multiplier.

Next, let $\mathcal{L}_G$ denote the positive Laplace-Beltrami operator corresponding to the unique (up to a constant) bi-invariant Riemmanian metric on $G$. Then for every $\xi\in[\xi]\in\widehat{G}$, its coefficient functions are eigenfunctions of $\mathcal{L}_G$ with same eigenvalue denoted by $\nu_\xi=\nu_{[\xi]}\geq 0$. Note that $\nu_\xi=0$ if and only if $\xi={\bf1_G}$, the trivial representation on $G$. Setting the weight
\begin{equation*}
    \langle \xi\rangle\vcentcolon=\sqrt{1+\nu_\xi},
\end{equation*}
we have that 
\begin{equation}\label{def_C1}
    \nu_\xi \leq \langle \xi \rangle^2 \leq C_1\nu_\xi,
\end{equation}
for every non-trivial $[\xi]\in\widehat{G}$,  where $C_1= 1+\frac{1}{\nu_\eta}$, and $\nu_\eta$ denotes the smallest non-zero eigenvalue of $\mathcal{L}_G$ (recall that its spectrum is discrete). 

For $s\in\R$, the Sobolev spaces $H^s(G)$ are then defined as the set of all $u\in\mathcal{D}'(G)$ with finite Sobolev norm:
\begin{equation*} \|u\|_{H^s(G)}\vcentcolon=\left(\sum_{[\xi]\in\widehat{G}}d_\xi \langle \xi\rangle^{2s} \|\widehat{u}(\xi)\|_{\HS}^2\right)^{\frac{1}{2}}.
\end{equation*}
We have that 
\begin{equation*}
     \bigcap_{s\in\R} H^s(G)=C^\infty(G) \text{ and } \bigcup_{s\in\R} H^s(G)=\mathcal{D}'(G),
\end{equation*}
and consequently the following characterization of smooth functions and distributions on $G$ holds:
\begin{prop}\label{prop_charact_smooth_and_dis}
    Let $({v}(\xi))_{[\xi]\in\widehat{G}}$ denote a sequence of matrices such that $v(\xi)\in \C^{d_\xi\times d_\xi}$ for every $[\xi]\in\widehat{G}$. Then $v(\cdot)=\widehat{f}(\cdot)$ for a smooth function $f\in C^\infty(G)$ if and only if for every $N>0$, there exists $C_N>0$ such that
    \begin{equation*}
        \|v(\xi)\|_{\HS}\leq C_N\langle \xi\rangle^{-N},
    \end{equation*}
    for every $[\xi]\in\widehat{G}$. Moreover, $v(\cdot)=\widehat{u}(\cdot)$ for a distribution $u\in \mathcal{D}'(G)$ if and only if there exist $C,N>0$ such that
    \begin{equation*}
         \|v(\xi)\|_{\HS}\leq C\langle \xi\rangle^{N},
    \end{equation*}
     for every $[\xi]\in\widehat{G}$. 
\end{prop}
For a more detailed exposition on the Fourier analysis of compact Lie groups, we refer to \cite{RT2010_book}.

Finally, for $m,n,d\in\N$, and a complex matrix $A\in \C^{m\times n}$, $m,n\in\N$, we fix the notation:
\begin{equation*}
    \ker A=\Big\{ u\in \C^{n}:Au=0\Big\}\quad \text{and}\quad \operatorname{ran}A=\Big\{v\in\C^{m}:\exists u\in \C^n,\ Au=v\Big\}.
\end{equation*}
Also, we shall denote by $\lambda_{\min}[A]$ its smallest singular value, and if $A\neq 0$, we denote by $\lambda_{\min}^{>0}[A]$ its smallest non-zero singular value. Note that $\lambda_{\min}^{>0}[A]$ is always positive and coincides with $\lambda_{\min}[A|_{(\ker A)^\perp}]$, where $A|_{(\ker A)^\perp}$ denotes the restriction of $A$ to the subspace $(\ker A)^\perp$. Hence, $\lambda_{\min}^{>0}[A]$ satisfies 
\begin{equation*}
    \lambda_{\min}^{>0}[A]=\min_{\substack{v\in (\ker A)^\perp\\\|v\|_2=1}} \|Av\|_2,
\end{equation*}
where $\|\cdot\|_{2}$ denotes the Euclidean norm. Consequently,
\begin{equation*}
    \|AB\|_{\HS}\geq \lambda_{\min}^{>0}[A]\|B\|_{\HS},
\end{equation*}
for any compatible matrix $B\in\C^{n\times d}$ whose columns are in $(\ker A)^\perp$. 

\section{Directional Poincaré inequality on compact Lie groups}

In this section we will prove necessary and sufficient conditions for the existence of a directional Poincaré inequality for left-invariant vector fields on a compact Lie group. From this characterization, we derive several corollaries (on general compact Lie groups and on the torus), and as an example we prove that every left-invariant vector field on $\S$ satisfies a directional Poincaré inequality with exponent $\delta=1$. But first, we fix some notation:

Let $G$ be a compact Lie group. There exists a basis of left-invariant vector fields $\mathcal{B}=\{X_1,\dots,X_n\}$ such that $\mathcal{L}_G=-(X_1^2+\dots+X_n^2)$. Denote by $\nabla_G:C^\infty(G)\to C^\infty(G)^n$ the gradient operator associated to the ordered basis $\mathcal{B}$. In other words,
\begin{equation*}
\nabla_Gf = (X_1f,\dots, X_nf),    
\end{equation*}
 for any $f\in C^\infty(G)$. Then any (real) left-invariant vector field $Y$ on $G$ can be written as $\langle \nabla_G,\alpha\rangle\vcentcolon=\alpha_1X_1+\dots +\alpha_n X_n$, where $\alpha\in \R^n$. When no confusion may arise, we shall sometimes omit the subscript $G$ and denote $\nabla_G$ simply by $\nabla.$

\subsection{Necessary and sufficient conditions}\label{sect_nec_suff}

\begin{prop}\label{prop_necessary}
     Let $Y=\langle \nabla_{G},\alpha\rangle$ be a left-invariant real vector field on a compact Lie group $G$ and $\delta\geq 1$. There exists $c>0$ such that 
    \begin{equation}\label{ineq_poincare_necessary}
        \|\nabla_{G} f\|_{L^2(G)}^{\delta-1}\| Yf\|_{L^2(G)}\geq c\|f\|_{L^2(G)}^\delta,
    \end{equation}
    for every $f\in H^1(G)$ with mean value zero, if and only if there exists $C>0$ such that
    \begin{equation}\label{ineq_necessary}
        \lambda_{\min}[\sigma_Y(\xi)]\geq C\langle \xi \rangle^{-(\delta-1)},
    \end{equation}
    for all non trivial $[\xi]\in\widehat{G}$. 
\end{prop}
\begin{proof}
For the sufficiency of the condition stated, note that if inequality \eqref{ineq_necessary} holds, then by \cite{strongly_invariant_KW,tese-Nicholas} the vector field $Y$ must be globally hypoelliptic. But then by \cite{GW_GH_vector_fields} we must have that $G$ is a torus. Therefore inequality \eqref{ineq_poincare_necessary} follows from a slight adaptation in the proof of \cite[Theorem 1]{directional_Poincare}.

    Next, suppose that \eqref{ineq_necessary} does not hold. First assume that  $\lambda_{\min}[\sigma_Y(\xi_0)]=0$ for some non-trivial  $[\xi_0]$. Let $f\in H^1(G)$ be given by: $\widehat{f}(\xi)$ is the zero matrix for all $[\xi]\neq [\xi_0]$, and $\widehat{f}(\xi_0)$ is given by 
\begin{equation*}
    \widehat{f}(\xi_0)=\begin{pmatrix}
      v_0 & 0 & \cdots & 0  
    \end{pmatrix}\in\C^{d_{\xi_0}\times d_{\xi_0}},
\end{equation*}
for some $v_0\in\ker \sigma_Y(\xi_0)$ with $\|v_0\|=1$. Clearly $f\in H^1(G)$ (in fact $f\in C^\infty(G)$), and since $0=\widehat{f}({\bf1}_G)=\int_Gf(x)dx$, where ${\bf1}_G$ denotes the trivial representation, we have that $f$ has mean value zero.

Then by Plancherel's Identity we have that $\|Yf\|_{L^2}=0$ and $\|f\|_{L^2}=1$, so that inequality \eqref{ineq_poincare_necessary} cannot hold for any $d>1$ and $c>0$.

Now, assume that $\lambda_{\min}[\sigma_Y(\xi)]\neq 0$, for every non-trivial $[\xi]\in\widehat{G}$. Then,  there exist sequences $[\xi_n]\in\widehat{G}$ and $v_{n}\in \C^{d_{\xi_n}}$, $\|v_n\|_2=1$, such that 
\begin{equation*}
  0< \| \sigma_Y(\xi_n)v_n\|_2<\frac{1}{n}\langle \xi_n \rangle ^{-(\delta-1)},
\end{equation*}
for every $n\in\N$.
 Let $f_n\in H^1(G)$ be given by
\begin{equation*}
    \widehat{f_n}(\xi)=\begin{cases}
        0&\text{if }\xi\neq \xi_n,\\
         \begin{pmatrix}
        v_n&0&\cdots&0
    \end{pmatrix}_{d_{\xi_n}\times d_{\xi_n}}&\text{if }\xi=\xi_n.
    \end{cases}
\end{equation*}
    Note that by Plancherel's identity we have that 
    \begin{align*}
        \|f_n\|_{L^2}^{2\delta}&=1,\\
        \|\nabla f_n\|_{L^2}^{2(\delta-1)}&=\nu_{\xi_n}^{\delta-1}\sim\langle \xi_n\rangle^{2(\delta-1)},\quad \text{and}\\
        \|Yf_n\|_{L^2}^{2}&= \|\sigma_Y(\xi_n)v_n\|_2^2< \frac{1}{n^2}\langle\xi_n\rangle^{-2(\delta-1)}.
    \end{align*} 
    Therefore,
    \begin{align*}
         \frac{ \|\nabla f_n\|_{L^2}^{2(\delta-1)}\| Yf_n\|_{L^2}^2}{\|f_n\|_{L^2}^{2\delta}}&\lesssim \frac{1}{n^2}\to 0,\quad \text{as } n\to \infty,
    \end{align*}
 so that inequality \eqref{ineq_poincare_necessary} cannot hold for any $c>0$, as claimed. 
\end{proof}

\begin{corollary}\label{coro_poincare_implies_torus}
Let $Y=\langle \nabla_{G},\alpha\rangle$ be a left-invariant real vector field on a compact Lie group $G$ and $\delta\geq 1$. If there exists $c>0$ such that 
    \begin{equation}\label{eq_poincare_implies_torus}
        \|\nabla_{G} f\|_{L^2(G)}^{\delta-1}\| Yf\|_{L^2(G)}\geq c\|f\|_{L^2(G)}^\delta,
    \end{equation}
    for every $f\in H^1(G)$ with mean-value zero, then $G$ is a torus.
\end{corollary}
\begin{proof}
Suppose that there exist $\delta\geq 1$ and $c>0$ such that \eqref{eq_poincare_implies_torus} holds for every $f\in H^1(G)$ with mean-value zero. Then by Proposition \ref{prop_necessary}, there exists $C>0$ such that
\begin{equation*}
    \lambda_{\min}[\sigma_{Y}(\xi)]\geq C\langle \xi\rangle^{-(\delta-1)},
\end{equation*}
for every non-trivial $[\xi]\in\widehat{G}$. By a result in \cite{strongly_invariant_KW,tese-Nicholas} this implies that $Y$ is a real globally hypoelliptic vector field on a compact Lie group, so by \cite{GW_GH_vector_fields} we must have that $G$ is (diffeomorphic to) a torus.
\end{proof}

\begin{definition}\label{def_EY}
     Let $s\in\R$ and $P:H^{m+s}(G)\to H^s(G)$ be a bounded linear operator. Denote by $ (\ker P)^\perp$ the subspace of $H^{m+s}(G)$ corresponding to the orthogonal complement of the kernel of $P$ with respect to the usual inner product $\langle \cdot ,\cdot \rangle _{H^{m+s}(G)}=\langle \Xi^{(m+s)/2}\cdot , \Xi^{(m+s)/2}\cdot\rangle_{L^2(G)}$, where $\Xi=\operatorname{Id}+\mathcal{L}_G$.
\end{definition}

Next, for left-invariant continuous linear operators $P:C^\infty(G)\to C^\infty(G)$ of order $m\in\R$, we characterize the distributions in the space $(\ker P|_{H^{m+s}(G)})^\perp\subset H^{m+s}(G)$, introduced in Definition \ref{def_EY}, in terms of conditions on their Fourier coefficients. Here, and for the rest of this paper, $P|_{H^{m+s}(G)}$ will denote the extension $P|_{H^{m+s}(G)}:H^{m+s}(G)\to H^s(G)$ of $P$ which is bounded by definition.

\begin{obs}\label{remark_perp}
    As already mentioned in the Introduction, in contrast to the directional Poincaré inequality on the torus introduced in  \cite{directional_Poincare}, we will prove the directional Poincaré inequality for all $f\in  (\ker Y|_{H^1(G)})^\perp\subset H^1(G)$ instead of for $f$ with mean value zero.  This choice is natural as it provides the largest subspace of $H^1(G)$ where this type of inequality can possibly hold, which can be seen directly from the fact that $H^1(G)=\ker Y|_{H^1(G)}\oplus(\ker Y|_{H^1(G)})^\perp$ and a directional Poincaré inequality cannot hold on a subspace $W\subset H^1(G)$ on which $Y:W\to L^2(G)$ is not injective.     
\end{obs}

\begin{prop}\label{prop_EP_equiv_Fourier}
    Let $s\in\R$ and $P:C^\infty(G)\to C^\infty(G)$ be left-invariant continuous linear operator of order $m\in\R$. A distribution $f\in H^{m+s}(G)$ is in  $(\ker P|_{H^{m+s}(G)})^\perp$ if and only if, for every $[\xi]\in\widehat{G}$, the columns of $\widehat{f}(\xi)$ are in $(\ker \sigma_P(\xi))^\perp$.
\end{prop}
\begin{proof}
    
Suppose that $f\in H^{m+s}(G)$ is in $(\ker P|_{H^{m+s}(G)})^{\perp}$. Fix $[\xi]\in\widehat{G}$ and let $v_\xi \in \ker \sigma_P({\xi})$. For $1\leq j\leq d_\xi$, let $u_j\in C^\infty(G)\subset {H^{m+s}(G)}$ be given by $\widehat{u_j}(\eta)=0$ for $[{\xi}]\neq [\eta]\in\widehat{G}$ and $\widehat{u_j}({\xi})$ be the $d_\xi\times d_\xi$ matrix  with $v_\xi$ as its $j$-th column and every other column be the zero vector. Then,
\begin{align*}
    Pu_j&=d_{\xi} \Tr(\sigma_{P}({\xi})\widehat{u_j}(\xi){\xi}(x))=d_{\xi}\Tr\left(\begin{pmatrix}
        0&\cdots&0& \sigma_{L}({\xi})v_\xi&0&\cdots&0
    \end{pmatrix}{\xi}(x)\right)=0;
\end{align*}
that is, $u_j\in \ker P|_{H^{m+s}(G)}$.
By identity \eqref{eq_Plancherel} we have that
    \begin{equation*}
        0=\frac{1}{d_\xi\langle \xi\rangle^{2(m+s)}} \langle f,u_j\rangle_{H^{m+s}(G)} = \Tr (\widehat{f}({\xi})\widehat{u_j}({\xi})^*)=\Tr (\overline{\widehat{u_j}({\xi})}^t\widehat{f}({\xi}))=\langle \widehat{f}({\xi})_{\cdot j},v_\xi\rangle_{\C^{d_\xi}},
    \end{equation*}
    where $\widehat{f}({\xi})_{\cdot j}$ denotes the $j$-th column of $\widehat{f}({\xi})$ and $\langle \cdot ,\cdot\rangle_{\C^{d_\xi}}$ denotes the usual inner product on $\C^{d_\xi}$. Hence every column of $\widehat{f}({\xi})$ is in $(\ker \sigma_{P}({\xi}))^\perp$, as claimed.
    
    Conversely, if  every column of $\widehat{f}(\xi)$ is in $(\ker \sigma_{P}(\xi))^\perp$, then for any $u\in \ker P|_{H^{m+s}(G)}$, we have that 
    \begin{align*}
        \langle f,u\rangle_{H^{m+s}(G)}&=\sum_{[\xi]\in\widehat{G}} d_\xi \langle \xi\rangle^{2(m+s)} \Tr(\widehat{f}(\xi)\widehat{u}({\xi})^*)=\sum_{[\xi]\in\widehat{G}} d_\xi \langle \xi\rangle^{2(m+s)}\Tr(\overline{\widehat{u}({\xi})}^t\widehat{f}(\xi))\\
        &=\sum_{[\xi]\in\widehat{G}} d_\xi\langle \xi\rangle^{2(m+s)}\sum_{j=1}^{d_\xi}\langle \widehat{f}({\xi})_{\cdot j},\widehat{u}({\xi})_{\cdot j}\rangle_{\C^{d_\xi}}=0,
    \end{align*}
    since every term in the sum above is equal to zero as $\widehat{u}({\xi})_{\cdot j}\in \ker \sigma_P(\xi)$, for every $[\xi]\in\widehat{G}$ and $1\leq j\leq d_\xi$. Therefore $f\in (\ker P|_{H^{m+s}(G)})^{\perp}$.
\end{proof}

\begin{theorem}\label{theo_directional_poincare_general}
    Let $Y=\langle \nabla_{G},\alpha\rangle$ be a left-invariant real vector field on a compact Lie group $G$ and $\delta\geq 1$. There exists $c>0$ such that 
\begin{equation}\label{ineq_poincare_nec_suf}
        \|\nabla_{G} f\|_{L^2(G)}^{\delta-1}\| Yf\|_{L^2(G)}\geq c\|f\|_{L^2(G)}^\delta,
    \end{equation}
    for every $f\in (\ker Y|_{H^1(G)})^\perp$, if and only if there exist $C>0$ such that
\begin{equation}\label{ineq_min_sin_nonzero_general_poincare}
        \lambda_{\min}^{>0}[\sigma_Y(\xi)]\geq C\langle \xi \rangle^{-(\delta-1)},
    \end{equation}
    for all $[\xi]\in\widehat{G}$ such that $\sigma_Y(\xi)\neq 0$. 
\end{theorem}
\begin{proof}
Note that if \eqref{ineq_min_sin_nonzero_general_poincare} does not hold, then there exist sequences $[\xi_n]\in\widehat{G}$ and $v_{n}\in (\ker \sigma_Y(\xi_n))^\perp\subset \C^{d_{\xi_n}}$, $\|v_n\|_2=1$, such that 
\begin{equation*}
  0< \| \sigma_Y(\xi_n)v_n\|_2<\frac{1}{n}\langle \xi_n \rangle ^{-(\delta-1)},
\end{equation*}
for every $n\in\N$.

Then, for every $n\in\N$, we have that the function $f_n\in H^1(G)$ given by the Fourier coefficients 
\begin{equation*}
    \widehat{f_n}(\xi)=\begin{cases}
        0&\text{if }\xi\neq \xi_n,\\
         \begin{pmatrix}
        v_n&0&\cdots&0
    \end{pmatrix}_{d_{\xi_n}\times d_{\xi_n}}&\text{if }\xi=\xi_n,
    \end{cases}
\end{equation*}
is in $(\ker Y|_{H^1(G)})^\perp$ by Proposition \ref{prop_EP_equiv_Fourier}. Following the same arguments as in the proof of Proposition \ref{prop_necessary}, we conclude that inequality \eqref{ineq_poincare_nec_suf} cannot hold.

Next assume that inequality \eqref{ineq_min_sin_nonzero_general_poincare} holds, and let $f\in (\ker Y|_{H^1(G)})^\perp$. We may assume $f\neq 0$, otherwise there is nothing to prove.

Note that if $\sigma_Y(\xi)\neq 0$, the smallest singular value of the linear operator $\sigma_Y(\xi):\C^{d_\xi}\to \C^{d_\xi}$ restricted to the subspace $(\ker \sigma_Y(\xi))^\perp$ is precisely $\lambda_{\min}^{>0}[\sigma_Y(\xi)]$. 
Since by Proposition \ref{prop_EP_equiv_Fourier}, the columns of $\widehat{f}(\xi)$ belong to $(\ker \sigma_Y(\xi))^\perp$ for every $[\xi]\in\widehat{G}$,  we conclude that 
\begin{equation}\label{ineq_symbol_norm_hs}
    \|\sigma_Y(\xi)\widehat{f}(\xi)\|_{\HS}\geq \lambda_{\min}^{>0}[\sigma_Y(\xi)]\|\widehat{f}(\xi)\|_{\HS},
\end{equation}
for every $[\xi]\in\widehat{G}$ such that $\sigma_Y(\xi)\neq 0$. Otherwise, if $\sigma_Y(\xi)=0$, then $\widehat{f}(\xi)=0$, again by Proposition \ref{prop_EP_equiv_Fourier},   and \eqref{ineq_symbol_norm_hs} also holds trivially in this case.

We conclude that
\begin{align}
    \|Yf\|_{L^2(G)}^2&=\sum_{[\xi]\in\widehat{G}}d_\xi\|\sigma_Y(\xi)\widehat{f}(\xi)\|_{\HS}^2\geq \sum_{[\xi]\in\widehat{G}}d_\xi(\lambda_{\min}^{>0}[\sigma_Y(\xi)])^2\|\widehat{f}(\xi)\|_{\HS}^2\notag\\
    &\geq \sum_{[\xi]\in\widehat{G}}d_\xi C^2\langle \xi \rangle^{-2(\delta-1)} \|\widehat{f}(\xi)\|_{\HS}^2.\label{ineq_poincare_suff_proof}
\end{align}
     Next note that for any non-zero $g\in H^1(G)$, we have
\begin{align*}
    \|\nabla g\|_{L^2(G)}^2&=\sum_{[\xi]\in\widehat{G}}d_\xi\nu_\xi\|\widehat{g}(\xi)\|_{\HS}^2\geq \sum_{\substack{\\
        \sqrt{\nu_\xi} \geq 2\frac{\|\nabla g\|_{L^2(G)}}{\|g\|_{L^2(G)}}}}d_\xi\nu_\xi\|\widehat{g}(\xi)\|_{\HS}^2\\
        &\geq 4\frac{\|\nabla g\|_{L^2(G)}^2}{\|g\|_{L^2(G)}^2}\sum_{\substack{\\
        \sqrt{\nu_\xi} \geq 2\frac{\|\nabla g\|_{L^2(G)}}{\|g\|_{L^2(G)}}}}d_\xi\|\widehat{g}(\xi)\|_{\HS}^2;
\end{align*}
hence,
\begin{equation*}
    \sum_{\substack{\\
        \sqrt{\nu_\xi} \geq 2\frac{\|\nabla g\|_{L^2(G)}}{\|g\|_{L^2(G)}}}}d_\xi\|\widehat{g}(\xi)\|_{\HS}^2\leq \frac{\|g\|_{L^2(G)}^2}{4}.
\end{equation*}
Consequently, 
    \begin{equation}\label{ineq_L^2_norm}
    \sum_{\substack{\\
        \sqrt{\nu_\xi} < 2\frac{\|\nabla g\|_{L^2(G)}}{\|g\|_{L^2(G)}}}}d_\xi\|\widehat{g}(\xi)\|_{\HS}^2\leq\frac{3\|g\|_{L^2(G)}^2}{4}. 
\end{equation}
 From  \eqref{ineq_poincare_suff_proof} we obtain
\begin{align*}
    \|Yf\|_{L^2(G)}^2&\geq C^2 \sum_{[\xi]\in\widehat{G}}d_\xi \langle \xi \rangle^{-2(\delta-1)} \|\widehat{f}(\xi)\|_{\HS}^2\geq  C^2\sum_{\substack{\\
        \sqrt{\nu_\xi} < 2\frac{\|\nabla f\|_{L^2(G)}}{\|f\|_{L^2(G)}}}}d_\xi\langle \xi \rangle^{-2(\delta-1)} \|\widehat{f}(\xi)\|_{\HS}^2\\
        &\geq \frac{C^2}{C_1^{\delta-1}}\sum_{\substack{\\
        \sqrt{\nu_\xi} < 2\frac{\|\nabla f\|_{L^2(G)}}{\|f\|_{L^2(G)}}}}d_\xi\nu_\xi^{-(\delta-1)} \|\widehat{f}(\xi)\|_{\HS}^2\\
        &\geq  \frac{C^2}{(4C_1)^{\delta-1}}\frac{\|f\|_{L^2(G)}^{2(\delta-1)}}{\|\nabla f\|_{L^2(G)}^{2(\delta-1)}}\sum_{\substack{\\
        \sqrt{\nu_\xi} < 2\frac{\|\nabla f\|_{L^2(G)}}{\|f\|_{L^2(G)}}}}d_\xi \|\widehat{f}(\xi)\|_{\HS}^2\\
        &\stackrel{\eqref{ineq_L^2_norm}}{\geq}  \frac{C^2}{(4C_1)^{\delta-1}}\frac{\|f\|_{L^2(G)}^{2(\delta-1)}}{\|\nabla f\|_{L^2(G)}^{2(\delta-1)}}\frac{3\|f\|_{L^2(G)}^2}{4},
\end{align*}
where on the third line we used the inequality 
\begin{equation*}
    \langle \xi \rangle \leq \sqrt{C_1\nu_\xi},
\end{equation*}
which holds for every non-trivial $[\xi]\in\widehat{G}$ for $C_1$ as in \eqref{def_C1}. Rearranging the inequality for $\|Yf\|_{L^2(G)}^2$ obtained above, and taking square roots on both sides, yields
\begin{equation*}
   \|\nabla f\|_{L^2(G)}^{\delta-1} \|Yf\|_{L^2(G)}\geq \frac{\sqrt{3}C}{(4\sqrt{C_1})^{\delta-1}}\|f\|_{L^2(G)}^\delta.
\end{equation*}
The proof is complete.
\end{proof}

Note that the proof of Theorem \ref{theo_directional_poincare_general}  still holds if we replace $Y$ by an arbitrary Fourier multiplier (left-invariant continuous linear operator acting on $C^\infty(G))$. Hence, we obtain the following corollary, whose proof is analogous:

\begin{corollary}\label{coro_poincare_multipliers}
     Let $G$ be a compact Lie group, $P:C^\infty(G)\to C^\infty(G)$ a  left-invariant continuous linear operator on a compact Lie group $G$ of order $m\in\R$. Given $\delta\geq 1$, there exists $c>0$ such that 
    \begin{equation*}
        \|\nabla_{G} f\|_{L^2(G)}^{\delta-1}\| Pf\|_{L^2(G)}\geq c\|f\|_{L^2(G)}^\delta,
    \end{equation*}
    for every $f\in (\ker P|_{H^{1}(G)})^\perp$, if and only if there exists $C>0$ such that
\begin{equation}\label{ineq_diof_multipliers}
        \lambda_{\min}^{>0}[\sigma_P(\xi)]\geq C\langle \xi \rangle^{-(\delta-1)},
    \end{equation}
    for all $[\xi]\in\widehat{G}$ such that $\sigma_P(\xi)\neq 0$.
\end{corollary}

\begin{prop}\label{prop_composed_ineq_general}
Let $P:C^\infty(G)\to C^\infty(G)$ be a left-invariant continuous linear operator of order $m\in\R$ which is also normal with respect to the $L^2$ inner product and such that \eqref{ineq_diof_multipliers} holds for some $C>0$ and $\delta\geq 1$. Then there exists  $c>0$ such that
        \begin{equation}\label{ineq_Y2}
            \|\nabla_G (Pu)\|_{L^2(G)}^{\delta-1}\|P^2u\|_{L^2(G)}\geq c\|Pu\|_{L^2(G)}^\delta,
        \end{equation}
        for every $u\in H^{m+1}(G)$. Moreover, if $u\in H^{m+1}(G)\cap (\ker P|_{H^{1}(G)})^\perp$, then 
         \begin{align}\label{ineq_Y22}
        \|\nabla_G u\|^{\delta(\delta-1)}_{L^2(G)} \|\nabla_G (Pu)\|_{L^2(G)}^{\delta-1}\|P^2u\|_{L^2(G)}&\geq c^{\delta+1}\|u\|_{L^2(G)}^{\delta^2}
    \end{align}
    also holds.
\end{prop}
\begin{proof}
To prove the first inequality, first note that since $P$ is normal with respect to the $L^2$ inner product, i.e. it commutes with its adjoint, its symbol $\sigma_P(\xi)$ is a  normal matrix for every $[\xi]\in\widehat{G}$. Consequently, we have that 
    \begin{equation*}
        (\ker \sigma_P(\xi))^\perp=\operatorname{ran}\sigma_P(\xi),
    \end{equation*}
    for every $[\xi]\in\widehat{G}$, and so for $u\in H^{m+1}(G)$ we have that
\begin{equation*}
    \widehat{Pu}(\xi)=\sigma_P(\xi)\widehat{u}(\xi)\in\operatorname{ran}\sigma_P(\xi)=(\ker \sigma_P(\xi))^\perp,
\end{equation*}
for every $[\xi]\in\widehat{G}$, which proves that $Pu\in (\ker P|_{H^{1}})^\perp$ by Proposition \ref{prop_EP_equiv_Fourier}  (since $P:H^{m+1}(G)\to H^{1}(G)$ is bounded). Therefore \eqref{ineq_Y2} follows from Corollary  \ref{coro_poincare_multipliers}.

 For the proof of inequality \eqref{ineq_Y22}, note that multiplying both sides of \eqref{ineq_Y2} by $ \|\nabla_G u\|^{\delta(\delta-1)}_{L^2(G)}$ and applying Corollary  \ref{coro_poincare_multipliers} once again (since $u\in (\ker P|_{H^{1}(G)})^\perp$) yields:
    \begin{align*}
        \|\nabla_G u\|^{\delta(\delta-1)}_{L^2(G)} \|\nabla_G (Pu)\|_{L^2(G)}^{\delta-1}\|P^2u\|_{L^2(G)}&\geq c\|\nabla_G u\|^{\delta(\delta-1)}_{L^2(G)} \|Pu\|_{L^2(G)}^\delta\\
        &\geq c\left(\|\nabla_G u\|^{\delta-1}_{L^2(G)} \|Pu\|_{L^2(G)}\right)^\delta\\
        &\geq c^{\delta+1}\|u\|_{L^2(G)}^{\delta^2}.
    \end{align*}
\end{proof}

Since any (real) left-invariant vector field is antisymmetric on $L^2(G)$, it is, in particular, normal. Therefore we obtain the following corollary.
\begin{corollary}
    Let $Y$ be a (real) left-invariant vector field on a compact Lie group $G$ such that \eqref{ineq_min_sin_nonzero_general_poincare} holds for some $C>0$ and $\delta\geq 1$. Then there exists  $c>0$ such that
        \begin{equation*}
            \|\nabla_G (Yu)\|_{L^2(G)}^{\delta-1}\|Y^2u\|_{L^2(G)}\geq c\|Yu\|_{L^2(G)}^\delta,
        \end{equation*}
        for every $u\in H^{2}(G)$. Moreover, if $u\in H^{2}(G)\cap (\ker Y|_{H^1(G)})^\perp=(\ker Y|_{H^2(G)})^\perp$, then 
         \begin{align*}
        \|\nabla_G u\|^{\delta(\delta-1)}_{L^2(G)} \|\nabla_G (Yu)\|_{L^2(G)}^{\delta-1}\|Y^2u\|_{L^2(G)}&\geq c^{\delta+1}\|u\|_{L^2(G)}^{\delta^2}
    \end{align*}
    also holds.
\end{corollary}

\subsubsection{The directional Poincaré inequality on the torus}

{As a consequence of Theorem \ref{theo_directional_poincare_general} we obtain the following corollaries which relate the directional Poincaré inequality and number theoretical properties of $\alpha$, as also evidenced  in \cite{directional_Poincare}.} We  will denote by $\T^n$ the $n$-dimensional torus, for some arbitrary $n\in\N$. Since the directional Poincaré inequality reduces to the usual Poincaré inequality when $n=1$, we will also assume that $n\geq 2$.

In parallel to the main result obtained in \cite{directional_Poincare}, we have the following corollary of Theorem \ref{theo_directional_poincare_general}:

\begin{corollary}\label{coro_irrational}
      For every $\delta\geq 1$, the set $\mathcal{B}_\delta\subset \R^n$ of directions such that for every $\alpha\in \mathcal{B}_\delta$ there exists $c_\alpha>0$ such that \begin{equation*}
        \|\nabla_{\T^n} f\|_{L^2(\T^n)}^{\delta-1}\| \langle \nabla_{\T^n} f,\alpha\rangle \|_{L^2(\T^n)}\geq c_\alpha\|f\|_{L^2(\T^n)}^\delta,
    \end{equation*}
    for every $f\in H^1(\T^n)$ with mean value zero, is given by 
\begin{equation*}
    \mathcal{B}_\delta=\Big\{\alpha\in\R^n : \exists C>0\ \operatorname{s. t.}\ |\langle \xi,\alpha\rangle |\geq C|\xi|^{-(\delta-1)},\ \operatorname{ for\ every }\ \xi\in\Z^n\backslash\{0\}\Big\}.
\end{equation*}
The set $\mathcal{B}_\delta$ is non-empty if and only if $\delta\geq n$, in which case it is uncountable. Moreover, these sets satisfy $\mathcal{B}_\delta\subset \mathcal{B}_{\delta'}$ for $\delta\leq \delta'$.
\end{corollary}
\begin{proof}
    If there exists $c_\alpha>0$ as in the statement, then $\alpha\in \mathcal{B}_\delta$ by Theorem \ref{theo_directional_poincare_general}, since in this case the symbol of  $Y=\langle \nabla_{\T^n},\alpha\rangle=\sum_{j=1}^n \alpha_j\partial_{x_j}$ is given by $\sigma_Y(\xi)=i\sum_{j=1}^n \alpha_j\xi_j=i\langle \xi,\alpha\rangle$, with the usual identification of $\widehat{\T^n}\sim\Z^n$. Conversely, if $\alpha\in\mathcal{B}_\delta$ , then the symbol of $Y$ as above vanishes only at $\xi=0$, hence $f\in H^1(\T^n)$ is in $(\ker Y|_{H^1(G)})^\perp$ if and only if $\widehat{f}(0)=0$, or equivalently, has mean-value zero. Therefore the existence of $c_\alpha$ also follows from Theorem \ref{theo_directional_poincare_general}. The inclusion $\mathcal{B}_{\delta}\subset \mathcal{B}_{\delta}$ if $\delta\leq \delta'$ is evident from the definition of the sets $\mathcal{B}_\delta$, and finally the fact that they are non-empty if and only if $\delta\geq n$ follows from Dirichlet's approximation theorem (see also Perron \cite{Perron}), as mentioned in \cite{directional_Poincare}.
\end{proof}

Theorem \eqref{theo_directional_poincare_general} also yields the following corollary, which allows for directions whose coordinates are linearly dependent over $\mathbb{Q}$:
\begin{corollary}\label{coro_rational}
     
     There exists a set $\mathcal{B}_1\subset \R^n$ such that for every  $\alpha\in\mathcal{B}_1$ there exists $c_\alpha>0$ such that
\begin{equation}\label{ineq_rational}
        \| \langle \nabla_{\T^n} f,\alpha\rangle \|_{L^2(\T^n)}\geq c_\alpha\|f\|_{L^2(\T^n)},
    \end{equation}
    for every $f\in H^1(\T^n)$, satisfying $\widehat{f}(\xi)=0$ whenever $\sum_{j=1}^n \alpha_j\xi_j=0$. More precisely, $\mathcal{B}_1$ is the set of all $\alpha\in \R^n$ such that its entries are all linearly dependent over $\mathbb{Q}$. Consequently, $\mathcal{B}_1$ is uncountable but Lebesgue-null. 
\end{corollary}
\begin{proof}
    By Theorem \ref{theo_directional_poincare_general} inequality \eqref{ineq_rational} holds for some $c>0$ if and only if the symbol of $Y=\langle \nabla_{\T^n},\alpha\rangle=\sum_{j=1}^n \alpha_j\partial_{x_j}$ satisfies inequality \eqref{ineq_min_sin_nonzero_general_poincare} for some $C>0$ and  $\delta=1$. 
    As mentioned in the proof of Corollary \eqref{coro_irrational}, the necessary and sufficient condition on the symbol given in Theorem \eqref{theo_directional_poincare_general} is, in this case ($G=\T^n$, $\delta=1$), equivalent to
    \begin{equation}\label{ineq_symbol_rational}
        \left|\sum_{j=1}^n \alpha_j\xi_j\right|\geq C,
    \end{equation}
     for every $\xi\in\Z^n$ such that $\sum_{j=1}^n \alpha_j\xi_j\neq 0$.  Suppose $\alpha\neq 0$ and assume first that the entries of $\alpha$ are linearly dependent over $\mathbb{Q}$. If $\alpha=0$ there is nothing to prove. For  $\alpha\neq 0$, we can write $\alpha=\lambda(\frac{p_1}{q_1},\dots,\frac{p_n}{q_n})$, for some  $p_j,q_j\in \Z$, $q_j\neq 0$, $1\leq j\leq n$, and $\lambda\in \R\backslash\{0\}$.
    Then
    \begin{align*}
         \left|\sum_{j=1}^n \alpha_j\xi_j\right|&=|\lambda|\left|\frac{p_1}{q_1}\xi_1+\dots+\frac{p_n}{q_n}\xi_n\right|=\left|\frac{\lambda}{q_1\cdots q_n}\right|\left|p_1q_2\cdots q_n\xi_1+\dots+p_nq_1\cdots q_{n-1} \xi_n\right|,
    \end{align*}
    for every $\xi\in\Z^n$. Note that the expression above is a product of the form $A\cdot B_\xi$, where $A\in \R\backslash\{0\}$ is a non-zero constant and $B_\xi$ is a non-negative integer which depends on $\xi$. Hence $B_\xi \geq 1$ and $ \left|\sum_{j=1}^n \alpha_j\xi_j\right|\geq \left|\frac{\lambda}{q_1\cdots q_n}\right|$ whenever $\sum_{j=1}^n \alpha_j\xi_j\neq 0$. Next assume that inequality \eqref{ineq_symbol_rational} holds, and assume that the entries of $\alpha$ are not linearly dependent over $\mathbb{Q}$. Without loss of generality we may assume that $\alpha_1$ and $\alpha_2$ are not linearly dependent over $\mathbb{Q}$, and so $\alpha_2/\alpha_1\not\in\mathbb{Q}$. Therefore, for every $\xi \in \Z^2\times \{0\}\subset \Z^n$ we have that
   \begin{align*}
       \left|\sum_{j=1}^n \alpha_j\xi_j\right|&=|\alpha_1|\left|\xi_1+\frac{\alpha_2}{\alpha_1}\xi_2\right|\geq C,
   \end{align*}
   an absurd by Dirichlet's approximation theorem since $\alpha\not\in\mathbb{Q}$.
\end{proof}

Recall that the irrationality measure of $\alpha\in\R$ is the infimum over all $\mu>0$ such that
\begin{equation}\label{ineq_irrationality_measure}
    \left|\alpha-\frac{\xi_1}{\xi_2}\right|\geq \frac{1}{(\xi_2)^{\mu}},
\end{equation}
for all $\xi_1,\xi_2\in\Z^2$ with $|\xi_2|$ sufficiently large. Evidently the irrationality measure of a rational number is $1$, and, in a famous result, Roth proved in \cite{Roth1955} that the irrationality measure of every irrational algebraic number is $2$. Furthermore, Liouville's Theorem tells us that if $\alpha$ is algebraic of degree $2$, then  \eqref{ineq_irrationality_measure} holds for $\mu=2$ as well.

\begin{corollary}\label{coro_t2}
      Let $\alpha=(\alpha_1,\alpha_2)\in\R^2$. There exists $c_\alpha>0$ such that     \begin{equation*}\label{eq_coro_torus_2}
        \|\nabla_{\T^2} f\|_{L^2(\T^2)}^{\delta-1}\| \langle \nabla_{\T^2} f,\alpha\rangle \|_{L^2(\T^2)}\geq c_\alpha\|f\|_{L^2(\T^2)}^\delta,
    \end{equation*}
     for some $\delta\geq 1$ and every $f\in H^1(\T^2)$ satisfying $\widehat{f}(\xi)=0$ whenever $\langle \xi,\alpha\rangle=0$, if and only if $Y=\langle \nabla_{\T^2},\alpha\rangle$ is globally solvable, or equivalently, 
    one of the following conditions holds:
    \begin{enumerate}
        \item[\it (1)] $\alpha_1=0$ or $\alpha_2=0$---in this case inequality \eqref{eq_coro_torus_2} holds for any $\delta\geq1$;
        \item[\it (2)] $\alpha_1\alpha_2\neq0$ and $\alpha_2/\alpha_1$  is rational or an irrational non-Liouville number; in this case:
        \begin{enumerate}
            \item[{\it (2.1)}] Inequality \eqref{eq_coro_torus_2} holds for any $\delta$ greater than the irrationality measure of $\alpha_2/\alpha_1$;
            \item[{\it (2.2)}] If $\alpha_2/\alpha_1$ is algebraic of degree $2$, then we can also take $\delta=2$.
        \end{enumerate}
    \end{enumerate}
\end{corollary}

Finally, comparing the necessary and sufficient condition in Theorem \ref{theo_directional_poincare_general} with the notion of global solvability with loss of derivatives in \cite{kk_loss}, we obtain the following corollary.

\begin{corollary}
   Let $\alpha\in\R^n$ and $Y=\langle\nabla_{\T^n},\alpha\rangle=\alpha_1\partial_{x_1}+\dots+\partial_n\partial_{x_n}$. Given $\delta\geq 1$, the following are equivalent.
    \begin{itemize}
        \item[{\it (1)}] For every  \( f \in 
        (\ker Y|_{H^1(\T^n)})^\perp\), there exists \( u \in H^{2 - \delta}(\mathbb{T}^n) \) such that \( Yu = f \)
        \item[{\it (2)}] There exists $c>0$ such that
        \begin{equation*}
            \|\nabla_G f\|_{L^2(\T^n)}^{\delta-1}\|\langle \nabla_{\T^n}f,\alpha\rangle\|_{L^2(\T^n)}\geq c\|f\|_{L^2(\T^n)}^\delta,
        \end{equation*}
        for every $f\in (\ker Y|_{H^1(G)})^\perp$.   
    \end{itemize}
\end{corollary}

\subsection{\texorpdfstring{Directional Poincaré Inequality on $\mathbb{S}^3$}{Directional Poincaré Inequality on the 3-sphere}}\label{section_sphere}

Following the notation in \cite{RT2010_book}, consider the basis of left-invariant vector fields $\{D_1,D_2,D_3\}$ on the compact Lie group  $\mathbb{S}^3
\cong SU(2)$, which satisfies $\mathcal{L}_{\mathbb{S}^3}=-(D_1^2+D_2^2+D_3^2)$, and let $\nabla_{\S}$ be the associated gradient operator.

Before proving our main result, we present the following technical lemma.

\begin{lemma}\label{lemma_tridiagonal_eig}
    Let $\ell\in\N_0$ and consider the one-parameter families of tridiagonal matrices given by 
    \begin{align*}
    C_1(\ell,\theta)=
    \begin{pmatrix}
        \sin(\theta)\frac{2\ell+1}{2} & \cos(\theta)\frac{1}{2}&\\
        \cos(\theta)\frac{2\ell+1}{2} &\sin(\theta)\frac{2\ell-1}{2}&\cos(\theta)\frac{2}{2}&\\
        & \cos(\theta)\frac{2\ell}{2}&\sin(\theta)\frac{2\ell-3}{2}&\cos(\theta)\frac{3}{2}& \\
        &&\ddots&\ddots&\ddots&\\
    &&& \cos(\theta)\frac{2}{2}&\sin(\theta)(-\frac{2\ell-1}{2})&\cos(\theta)\frac{2\ell+1}{2}\\
        & &&& \cos(\theta)\frac{1}{2}&\sin(\theta)(-\frac{2\ell+1}{2})
    \end{pmatrix},\\\\
 C_2(\ell,\theta)=
    \begin{pmatrix}
        \sin(\theta){\ell} & \cos(\theta)\frac{1}{2}&\\
        \cos(\theta)\frac{2\ell}{2} &\sin(\theta)(\ell-1)&\cos(\theta)\frac{2}{2}&\\
        & \cos(\theta)\frac{2\ell-1}{2}&\sin(\theta)(\ell-2)&\cos(\theta)\frac{3}{2}& \\
        &&\ddots&\ddots&\ddots&\\
    &&& \cos(\theta)\frac{2}{2}&\sin(\theta)(-(\ell-1))&\cos(\theta)\frac{2\ell}{2}\\
        & &&& \cos(\theta)\frac{1}{2}&\sin(\theta)(-\ell)
    \end{pmatrix},
\end{align*}
where $\theta\in\R$. Then their eigenvalues are independent of the parameter $\theta\in\R$ and given by 
    \begin{equation*}
     \{\pm 1/2,\pm 3/2,\dots,\pm (2\ell+1)/2\}
    \end{equation*}
 and 
     \begin{equation*}
     \{0,\pm 1,\pm 2,\dots,\pm \ell\},
    \end{equation*}
     respectively.
\end{lemma}
\begin{proof} 
    Fix $\ell\in\N$, and let $X=iC_1(\ell,0)$ and $Y=i C_1(\ell,\pi/2)$. Note that we can write
    \begin{equation*}
        iC_1(\ell,\theta)=\cos(\theta)X+\sin(\theta)Y,
    \end{equation*}
    for every $\theta\in\R$.
    Now consider the commutator
    \begin{equation*}
        Z\defeq[X,Y]=XY-YX.
    \end{equation*}
    A simple computation shows that
    \begin{equation*}
        Z=\begin{pmatrix}
        0 & \frac{1}{2}&\\
        -\frac{2\ell+1}{2} &0&\frac{2}{2}&\\
        & -\frac{2\ell}{2}&&\frac{3}{2}& \\
        &&\ddots&\ddots&\ddots&\\
    &&& -\frac{2}{2}&0&\frac{2\ell+1}{2}\\
        & &&& -\frac{1}{2}&0
    \end{pmatrix}.
    \end{equation*}
Using this fact, it is easy to compute the commutators:
\begin{equation*}
    [Y,Z]=X\quad\text{and}\quad [X,Z]=-Y.
\end{equation*}
    Next consider the matrices
    \begin{equation*}
        M(\theta)=e^{\theta Z}Xe^{-\theta Z},
    \end{equation*}
    which are all conjugate to $X$ by definition. We claim that $iC_1(\ell,\theta)=M(\theta)$. Indeed, notice that the function $M(\cdot)$ satisfies the first-order differential equation $\frac{d}{d\theta}M(\theta)=[Z,M(\theta)]$ with initial condition $M(0)=X$. On the other hand
    \begin{align*}
        [Z,iC_1(\ell,\theta)]&=[Z,\cos(\theta)X+\sin(\theta)Y]=-\cos(\theta)[X,Z]-\sin(\theta)[Y,Z]\\
        &=\cos(\theta)Y-\sin(\theta)X\\
        &=\frac{d}{d\theta}\left(\cos(\theta)X+\sin(\theta)Y\right)\\
        &=\frac{d}{d\theta}iC_1(\ell,\theta);
    \end{align*}
    hence, $iC_1(\ell,\theta)$ also satisfies the differential equation. Since $iC_1(\ell,0)=X$, the same initial condition holds. Thus, by uniqueness of solution, we conclude the $iC_1(\ell,\theta)=M(\theta)$, for all $\theta$. Since $iC_1(\ell,\theta)$ is conjugate $X$ for all $\theta$, by transitivity we conclude that $i C_1(\ell,\theta)$ is conjugate to $Y=iC_1(\ell,\pi/2)$, which is diagonal with eigenvalues  $i\cdot\{\pm 1/2,\pm 3/2,\dots,\pm (2\ell+1)/2\}$. The proof for $C_2(\ell,\theta)$ is analogous.
\end{proof}

\begin{theorem}[Directional Poincaré inequality on $\S$]\label{theo_directional_sphere}
    For any $\alpha\in\R^3$ we have that
    \begin{equation}\label{ineq_Poincaré_S3}
        \|\langle \nabla_{\mathbb{S}^3}f,\alpha\rangle \|_{L^2(\mathbb{S}^3)}\geq \frac{\|\alpha\|}{2}\|f\|_{L^2(\mathbb{S}^3)},
    \end{equation}
    for every $f\in (\ker Y|_{H^1(G)})^\perp$, where $Y=\alpha_1D_1+\alpha_2D_2+\alpha_3D_3$.
\end{theorem}
\begin{proof}
Fix $\alpha\in\R^3$ and let $Y\defeq\langle\nabla_{\mathbb{S}^3},\alpha\rangle =\alpha_1D_1+\alpha_2D_2+\alpha_3D_3$, so that $\langle \nabla_{\mathbb{S}^3}f,\alpha\rangle=Yf$. 
The result will follow from Theorem \ref{theo_directional_poincare_general} once we prove that 
\begin{equation}\label{diof_condition_S^3}
    \lambda_{\min}^{>0}[\sigma_Y(\ell)]\geq \frac{\|\alpha\|}{2},
\end{equation}
for every $\ell\in\frac{1}{2}\N_0$ such that $\sigma_Y(\ell)\neq 0$.\\
For $\alpha=0$, the result is trivial. Assume that  $\alpha\neq 0$. The inequality \eqref{diof_condition_S^3} follows from \cite[Proposition 12.2.4]{RT2010_book}---keeping in mind the result there holds for vectors of same length in $\operatorname{Lie}(SU(2))\sim \operatorname{Lie}(\mathbb{S}^3)$---as similar matrices have the same eigenvalues, however we choose to give a more direct proof, as follows. First note that 
\begin{equation*}
\|Y f\|_{L^2(\mathbb{S}^3)}=   \|\alpha\|\cdot\|{\textstyle\frac{1}{\|\alpha\|}}Y f\|_{L^2(\mathbb{S}^3)},    
\end{equation*}
so we only need to prove that \eqref{diof_condition_S^3} holds for $\|\alpha\|=1$.  

 Next set $\alpha_*=(i\alpha_1-\alpha_2)/2$. Using the formulas for the symbols of $D_1,D_2$ and $D_3$ in \cite{RT2010_book}, we obtain that  $\sigma_Y(\ell)$ is given by the tridiagonal matrix
{\small\begin{equation*}
    \begin{pmatrix}
        i\alpha_3\ell & -\overline{\alpha_*}\sqrt{2\ell}&&&\\
        \alpha_*\sqrt{2\ell} &i\alpha_3(\ell-1)&-\overline{\alpha_*}\sqrt{(2\ell-1)2}&&\\
        &\alpha_*\sqrt{(2\ell-1)2}&i\alpha_3(\ell-2)&-\overline{\alpha_*}\sqrt{(2\ell-2)3} \\
        &&\ddots&\ddots&\ddots\\
        &&&\alpha_*\sqrt{2(2\ell-1)}&i\alpha_3(-(\ell-1))&-\overline{\alpha_*}\sqrt{2\ell}\\
        && &&\alpha_*\sqrt{2\ell}&i\alpha_3(-\ell)
    \end{pmatrix},
\end{equation*}}
for every $\ell\in\frac{1}{2}\N_0\sim \widehat{\mathbb{S}^3}$. Since $Y$ is anti-symmetric on $L^2(\S)$, we have that  $\sigma_Y(\ell)$ is anti-hermitian for every $\ell\in\frac{1}{2}\N_0$ (see \cite[Remark 10.4.20]{RT2010_book}); hence its singular values coincide with the absolute value of its eigenvalues. Therefore, it is enough to compute the eigenvalues of $\sigma_Y(\ell)$.

Note that if $\alpha_1=\alpha_2=0$, then $\sigma_Y(\ell)$ is diagonal and, consequently, its eigenvalues are given trivially by $i\alpha_3\{\pm 1/2,\pm 3/2,\dots,\pm\ell\}$; hence
\begin{equation*}
    \lambda_{\min}^{>0}[\sigma_Y(\ell)]=\frac{|\alpha_3|}{2}=\frac{\|\alpha\|}{2}=\frac{1}{2}.
\end{equation*}
Also, if $\ell=0$, $\sigma_Y(\ell)=\begin{pmatrix}
    0
\end{pmatrix}$,
and there is nothing to prove.
Next assume that $(\alpha_1,\alpha_2)\neq 0$ and $\ell\neq 0$.
First we will study the half-integer frequencies $\ell\in\frac{1}{2}\N_0\backslash\N_0$.

For $\ell\in\frac{1}{2}\N_0\backslash\N_0$, by considering the change of variables $\ell'=\ell-\frac{1}{2}$, we have that $ \tilde \sigma_Y(\ell')\vcentcolon=\frac{1}{i}\sigma_Y(\frac{2\ell'+1}{2})=\frac{1}{i}\sigma_Y(\ell)$ is equal to
\begin{equation*}
   \begin{pmatrix}
        \alpha_3(\frac{2\ell'+1}{2}) & -\overline{i\alpha_*}\sqrt{2\ell'+1}&&&\\
        -i\alpha_*\sqrt{2\ell'+1} &\alpha_3(\frac{2\ell'-1}{2})&-\overline{i\alpha_*}\sqrt{(2\ell')2}&&\\
        &-i\alpha_*\sqrt{(2\ell')2}&\alpha_3(\frac{2\ell'-3}{2})&-\overline{i\alpha_*}\sqrt{(2\ell-1)3}& \\
        &\ddots&\ddots&\ddots&\\
        &&-i\alpha_*\sqrt{(2\ell')2}&\alpha_3(-\frac{2\ell'-1}{2})&-\overline{i\alpha_*}\sqrt{2\ell'+1}\\
        & &&-i\alpha_*\sqrt{2\ell'+1}&\alpha_3(-\frac{2\ell'+1}{2}),
    \end{pmatrix}
\end{equation*}
for every $\ell'\in\N_0$, where $\tilde \sigma_Y(\ell')\in \C^{2\ell'+2\times 2\ell'+2}$. Henceforth we will consider $\ell\in\N_0$.

Let $D(\ell)= \operatorname{diag}(d_1,\dots,d_{2\ell+2})$, where
\begin{equation*}
    d_k=\left(\frac{2}{\sqrt{\alpha_1^2+\alpha_2^2}}\right)^{k-1}\frac{1}{(k-1)!}\prod_{j=1}^{k-1}\tilde \sigma_{Y}(\ell)_{j(j+1)},
\end{equation*}
with $d_1=1$. Setting $C_0(\ell,\alpha) = D(\ell)\tilde \sigma_{Y}(\ell)D(\ell)^{-1}$, a simple computation shows that 
\begin{equation*}
    C_0(\ell,\alpha)=\begin{pmatrix}
        \alpha_3(\frac{2\ell+1}{2}) & \frac{\|\alpha'\|}{2}1&\\
        \frac{\|\alpha'\|}{2}{(2\ell+1)} &\alpha_3(\frac{2\ell-1}{2})&\frac{\|\alpha'\|}{2}2&\\
        & \frac{\|\alpha'\|}{2}2\ell&\alpha_3(\frac{2\ell-3}{2})&\frac{\|\alpha'\|}{2}3& \\
        &&\ddots&\ddots&\ddots&\\
    &&& \frac{\|\alpha'\|}{2}2&\alpha_3(-\frac{2\ell-1}{2})&\frac{\|\alpha'\|}{2}(2\ell+1)\\
        & &&& \frac{\|\alpha'\|}{2}1&\alpha_3(-\frac{2\ell+1}{2})
    \end{pmatrix}.
\end{equation*}
Indeed, note that 
\begin{align*}
    C_0(\ell,\alpha)_{k(k+1)}&=D(\ell)_{kk}\tilde \sigma_{Y}(\ell)_{k(k+1)}D(\ell)^{-1}_{(k+1)(k+1)}=\frac{d_k}{d_{k+1}}\tilde \sigma_{Y}(\ell)_{k(k+1)}=\frac{\sqrt{\alpha_1^2+\alpha_2^2}}{2}k,
\end{align*}
 and
\begin{align*}
    C_0(\ell,\alpha)_{(k+1)k}&=D(\ell)_{(k+1)(k+1)}\tilde \sigma_{Y}(\ell)_{(k+1)k}D(\ell)^{-1}_{kk}\\
    &=\frac{d_{k+1}}{d_k}\tilde \sigma_{Y}(\ell)_{(k+1)k}\\
    &=\left(\frac{\sqrt{\alpha_1^2+\alpha_2^2}}{2}\right)^{-1}\frac{1}{k}\tilde \sigma_{Y}(\ell)_{(k+1)k}\tilde \sigma_{Y}(\ell)_{k(k+1)}\\
    &=\frac{2}{\sqrt{\alpha_1^2+\alpha_2^2}}\frac{1}{k}\frac{{\alpha_1^2+\alpha_2^2}}{4}(2\ell+2-k)k\\
    &=\frac{\sqrt{\alpha_1^2+\alpha_2^2}}{2}(2\ell+2-k),
\end{align*}
for $1\leq k\leq 2\ell+1$. Also,
\begin{align*}
     C_0(\ell,\alpha)_{kk}&=D(\ell)_{kk}\tilde \sigma_{Y}(\ell)_{kk}D(\ell)^{-1}_{kk}\\
     &=\tilde \sigma_{Y}(\ell)_{kk}\\
     &=\alpha_3 \left(\frac{2\ell+3-2k}{2}\right),
\end{align*}
for $1\leq k\leq 2\ell+2$.

Since $\|\alpha\|=1$, we can write $\alpha=(\cos\phi\cos \theta,\sin\phi\cos\theta,\sin\phi)$, where $0\leq \phi\leq 2\pi$ and $-\pi/2\leq \theta<\pi/2$. Then $C_0(\ell,\alpha)$ can be rewritten as\begin{align*}
    C_1(\ell,\theta)=\begin{pmatrix}
        \sin(\theta)\frac{2\ell+1}{2} & \cos(\theta)\frac{1}{2}&\\
        \cos(\theta)\frac{2\ell+1}{2} &\sin(\theta)\frac{2\ell-1}{2}&\cos(\theta)\frac{2}{2}&\\
        & \cos(\theta)\frac{2\ell}{2}&\sin(\theta)\frac{2\ell-3}{2}&\cos(\theta)\frac{3}{2}& \\
        &&\ddots&\ddots&\ddots&\\
    &&& \cos(\theta)\frac{2}{2}&\sin(\theta)(-\frac{2\ell-1}{2})&\cos(\theta)\frac{2\ell+1}{2}\\
        & &&& \cos(\theta)\frac{1}{2}&\sin(\theta)(-\frac{2\ell+1}{2})
    \end{pmatrix}.
\end{align*}
By Lemma \ref{lemma_tridiagonal_eig} the eigenvalues of $C_1(\ell,\theta)$ are given by 
\begin{equation*}
 \{\pm 1/2,\pm 3/2,\dots,\pm (2\ell+1)/2\}.  
\end{equation*}
So these also correspond to the eigenvalues of $C_0(\ell,\alpha)$  and by similarity $\tilde\sigma_Y(\ell)$ also has the same eigenvalues. We conclude that for $\ell\in\frac{1}{2}\N_0\backslash\N_0$, the eigenvalues of $\sigma_Y(\ell)$ are given by 
\begin{equation*}
    i \{\pm 1/2,\pm 3/2,\dots,\pm \ell\},
\end{equation*}
so its smallest non-zero singular value corresponds to $\frac{1}{2}$. 

For the case $\ell\in\N$, we follow an analogous argument, but this time we skip the change of variables $\ell'\mapsto \ell$ and apply the conjugation by $D(\ell)$ directly. We then write $\alpha$ in spherical coordinates again, and obtain that $\tilde\sigma_Y(\ell)$ is similar to 
\begin{align*}
    C_2(\ell,\theta)=\begin{pmatrix}
        \sin(\theta){\ell} & \cos(\theta)\frac{1}{2}&\\
        \cos(\theta)\frac{2\ell}{2} &\sin(\theta)(\ell-1)&\cos(\theta)\frac{2}{2}&\\
        & \cos(\theta)\frac{2\ell-1}{2}&\sin(\theta)(\ell-2)&\cos(\theta)\frac{3}{2}& \\
        &&\ddots&\ddots&\ddots&\\
    &&& \cos(\theta)\frac{2}{2}&\sin(\theta)(-(\ell-1))&\cos(\theta)\frac{2\ell}{2}\\
        & &&& \cos(\theta)\frac{1}{2}&\sin(\theta)(-\ell)
    \end{pmatrix}.
\end{align*}
Then, again by Lemma \ref{lemma_tridiagonal_eig}, we obtain that the eigenvalues of $C_2(\ell,\theta)$ are given by 
\begin{equation*}
    \{0,\pm 1,\dots,\pm \ell\},
\end{equation*}
so we conclude that the eigenvalues of $\sigma_Y(\ell)$ correspond to 
\begin{equation*}
   i \{0,\pm 1,\dots,\pm \ell\};
\end{equation*}
 hence, its smallest non-zero singular value is $1$. Therefore, we conclude that inequality \eqref{diof_condition_S^3} holds for every $\ell\in\frac{1}{2}\N_0$ such that $\sigma_Y(\ell)\neq0$. The proof is complete.
\end{proof}

\section{Global Solvability of Fourier Multipliers on compact Lie groups}\label{sect_solvability}

Global solvability of left-invariant continuous linear operators on compact Lie groups  has been characterized in terms of $L^2$ lower bounds on the action of the symbol of the operator on certain subspaces of the eigenspaces of the Laplacian (see \cite{araujo_compact,strongly_inv_solvable}), and, for certain classes of operators, in terms of algebraic conditions involving the eigenvalues of their respective global symbol (see \cite{diagonal_systems}). 

In this section we characterize the global solvability of left-invariant continuous linear operators on compact Lie groups in terms of lower bounds on the decay of the smallest non-zero singular value of the symbol. Consequently,  together with the results from Section \ref{sect_nec_suff}, we obtain that the global solvability of a left-invariant continuous linear operator is equivalent to the existence of its corresponding directional Poincaré inequality, for some exponent $\delta\geq 1$.

\begin{definition}
    Let $M$ be a compact manifold and $P:C^\infty(M)\to C^\infty(M)$ be  a continuous linear operator. We say that $P$ is globally solvable if for every $f\in(\ker {}^tP)^{0}$ there exists $u\in C^\infty(M)$ such that $Pu=f$, where 
    \begin{equation*}
        (\ker {}^tP)^{0}=\{f\in C^\infty(M):\langle u,f\rangle=0,\,\forall u \in \ker {}^tP\subset \mathcal{D}'(M)\}.
    \end{equation*}
\end{definition}
We note that since $\overline{\operatorname{ran} P}=(\ker {}^tP)^{0}$, an operator $P$ is globally solvable if and only if its range is closed in $C^\infty(M)$.

\begin{prop}\label{prop_kernel_annihi_range}
    Let $G$ be a compact Lie group and $P$ a left-invariant continuous linear operator on $C^\infty(G)$. Then $f\in C^\infty(G)$ is in $(\ker {}^tP)^{0}$ if and only if, for every $[\xi]\in\widehat{G}$, 
    every column of $\widehat{f}(\xi)$ is in  $\operatorname{ran}\sigma_P(\xi)$.
\end{prop}
\begin{proof}
    First note that $\sigma_{{}^tP}(\xi)$, the symbol of ${}^tP$ at $[\xi]\in\widehat{G}$, is given by $\sigma_P(\overline{\xi})^{t}$. Indeed, for $u\in\mathcal{D}'(G)$ and $f\in C^\infty(G)$, we have that
\begin{align}\label{eq_Fourier_duality}
    \langle u ,f \rangle&=\left\langle u,\sum_{[\xi]\in\widehat{G}}d_\xi \Tr\left(\widehat{f}(\xi)\xi(x)\right)\right\rangle=\sum_{[\xi]\in\widehat{G}}d_\xi \Tr\left(\widehat{f}(\xi)\langle u,\xi(x)\rangle\right)\notag\\
    &=\sum_{[\xi]\in\widehat{G}}d_\xi \Tr\left(\widehat{f}(\xi)^t\langle u,{\xi}(x)^t\rangle\right)=\sum_{[\xi]\in\widehat{G}}d_\xi \Tr\left(\widehat{f}(\xi)^t\widehat{u}(\overline{\xi})\right)\notag\\
    &=\sum_{[\xi]\in\widehat{G}}d_\xi \Tr\left(\widehat{f}(\xi)\widehat{u}(\overline{\xi})^t\right),
\end{align}
where we used the fact that the Fourier series of $f$ converges in $C^\infty(G)$.
Therefore, 
\begin{align*}
    \langle u,Pf\rangle=\sum_{[\xi]\in\widehat{G}}d_\xi \Tr\left(\sigma_{P}(\xi)\widehat{f}(\xi)\widehat{u}(\overline{\xi})^t\right),
\end{align*}
but also
\begin{align*}
     \langle {}^tPu,f\rangle&=\sum_{[\xi]\in\widehat{G}}d_\xi \Tr\left(\widehat{f}(\xi)\widehat{u}(\overline{\xi})^t\sigma_{{}^tP}(\overline{\xi})^t\right)\\
     &=\sum_{[\xi]\in\widehat{G}}d_\xi \Tr\left(\sigma_{{}^tP}(\overline{\xi})^t\widehat{f}(\xi)\widehat{u}(\overline{\xi})^t\right).
\end{align*}  
Since this holds for any $u\in\mathcal{D}'(G)$ and $f\in C^\infty(G)$, we must have that
\begin{equation*}
    \sigma_{{}^tP}(\xi)=\sigma_{P}(\overline{\xi})^{t},
\end{equation*}
for every $[\xi]\in\widehat{G}$, as claimed.

Therefore, $u\in \mathcal{D}'(G)$ is in $\ker {}^tP$ if and only if the colmuns of $\widehat{u}(\xi)$ are in $\ker\sigma_P(\overline{\xi})^t$, for every $[\xi]\in\widehat{G}$.

Next suppose that $f\in C^\infty(G)$ is in $(\ker {}^tP)^{0}$. Fix $[\xi]\in\widehat{G}$ and let $v_\xi \in \ker \sigma_P({\xi})^t$. For $1\leq j\leq d_\xi$, let $u_j\in \mathcal{D}'(G)$ be given by $\widehat{u_j}(\eta)=0$ for $[\overline{\xi}]\neq [\eta]\in\widehat{G}$ and $\widehat{u_j}(\overline{\xi})$ be the $d_\xi\times d_\xi$ matrix  with $v_\xi$ as its $j$-th column and every other column be the zero vector. Then,
\begin{align*}
    {}^tPu_j&=d_{\overline\xi} \Tr(\sigma_{{}^tP}(\overline{\xi})\widehat{u_j}(\overline\xi)\overline{\xi}(x))\\
    &=d_{\overline{\xi}} \Tr(\sigma_{P}({\xi})^t\widehat{u_j}(\overline{\xi})\overline{\xi}(x))\\
    &=d_{\overline{\xi}}\Tr\left(\begin{pmatrix}
        0&\cdots&0& \sigma_{P}({\xi})^tv_\xi&0&\cdots&0
    \end{pmatrix}\overline{\xi}(x)\right)=0;
\end{align*}
that is, $u_j\in \ker {}^tP$.
Also, by identity \eqref{eq_Fourier_duality} we have that
    \begin{equation*}
        0= \frac{1}{d_\xi}\langle u_j,f\rangle = \Tr (\widehat{f}({\xi})\widehat{u_j}(\overline{\xi})^t)=\Tr (\widehat{u_j}(\overline{\xi})^t\widehat{f}({\xi}))=\langle \widehat{f}({\xi})_{\cdot j},v_\xi\rangle_{\C^{d_\xi}},
    \end{equation*}
    where $\widehat{f}({\xi})_{\cdot j}$ denotes the $j$-th column of $\widehat{f}({\xi})$. Hence every column of $\widehat{f}({\xi})$ is in $(\ker \sigma_{P}({\xi})^t)^\perp=\operatorname{ran}\sigma_P({\xi})$, as claimed.
    
    Conversely, if  every column of $\widehat{f}(\xi)$ is in $\operatorname{ran}\sigma_P(\xi)=(\ker \sigma_{P}(\xi)^t)^\perp$, then for any $u\in \ker {}^tP$, we have that 
    \begin{align*}
        \langle u,f\rangle&=\sum_{[\xi]\in\widehat{G}} d_\xi \Tr(\widehat{f}(\xi)\widehat{u}(\overline{\xi})^t)
        =\sum_{[\xi]\in\widehat{G}} d_\xi\sum_{j=1}^{d_\xi}\langle \widehat{f}({\xi})_{\cdot j},\widehat{u}(\overline{\xi})_{\cdot j}\rangle_{\C^{d_\xi}}=0,
    \end{align*}
    since every term in the sum above is equal to zero as $\widehat{u}(\overline{\xi})_{\cdot j}\in \ker \sigma_{{}^tP}(\overline{\xi})=\ker \sigma_P(\xi)^t$, for every $[\xi]\in\widehat{G}$ and $1\leq j\leq d_\xi$. Therefore $f\in (\ker {}^tP)^{0}$.
\end{proof}

\begin{prop}\label{prop_globally_solvable}
    Let $G$ be a compact Lie group and $P:C^\infty(G)\to C^\infty(G)$ be a left-invariant continuous linear operator on $G$. Then $P$ is globally solvable (has closed range)  if and only if there exist $C,k>0$ such that
    \begin{equation}\label{ineq_diof_gs}
        \lambda_{\min}^{>0}[\sigma_P(\xi)]\geq C\langle \xi\rangle^{-k},
    \end{equation}
    for every $[\xi]\in\widehat{G}$ such that $\sigma_P(\xi)\neq 0$.
\end{prop}
\begin{proof}
Suppose that the inequality \eqref{ineq_diof_gs} is satisfied, and let $f\in C^\infty(G)\cap (\ker {}^tP)^0$.  We will construct $u\in C^\infty(G)$ such that $Pu=f$, as follows.

Let $[\xi]\in\widehat{G}$. If $\sigma_P(\xi)=0$, we must have $\widehat{f}(\xi)=0$. Then, let $\widehat{u}(\xi)=0$. 

Now assume $\sigma_P(\xi)\neq 0$. Due to the decomposition 
\begin{equation*}
    \ker \sigma_P(\xi)\oplus( \ker \sigma_P(\xi))^\perp=\C^{d_\xi},
\end{equation*}
by conjugating $\xi$ with a unitary change of basis (which amounts to choosing a different representative of the class $[\xi])$, we may assume that $\sigma_P(\xi)$ is given in block form
\begin{equation*}
    \sigma_P(\xi)=\begin{pmatrix}
        \tilde \sigma_P(\xi)&0\\
        0&0
    \end{pmatrix},
\end{equation*}
where the matrix $\tilde \sigma_P(\xi)$ corresponds to the restriction
\begin{equation*}
    \tilde \sigma_P(\xi)= \sigma_P(\xi)|_{(\ker \sigma_P(\xi))^\perp}:(\ker \sigma_P(\xi))^\perp\to \operatorname{ran}\sigma_P(\xi),
\end{equation*}
is invertible and satisfies
\begin{equation*}
     \lambda_{\min}[\tilde\sigma_P(\xi)]= \lambda_{\min}^{>0}[\sigma_P(\xi)]\geq C\langle \xi \rangle^{-k}.
\end{equation*}
This also implies that 
\begin{equation*}
    \| \tilde \sigma_P(\xi)^{-1}\|_{\operatorname{op}}\leq C^{-1}\langle \xi \rangle^{k},
\end{equation*}
and since the columns of $\widehat{f}(\xi)$ are in the range of $\sigma_P(\xi)$, by Proposition \ref{prop_kernel_annihi_range}, we must have that $\widehat{f}(\xi)$ is given in block form by
\begin{equation*}
    \widehat{f}(\xi)= \begin{pmatrix}
        {\widehat{f_1}}(\xi)\\0
    \end{pmatrix},
\end{equation*}
where $\widehat{f_1}(\xi)\in\C^{d^0_\xi\times d_\xi}$, $d^0_\xi = \dim \operatorname{ran}\sigma_P(\xi)=\dim (\ker \sigma_P(\xi))^{\perp}$.

Therefore, defining
\begin{equation*}
    \widehat{u}(\xi)=\begin{pmatrix}
        \tilde \sigma_P(\xi)^{-1}\widehat{f_1}(\xi)\\
        0
    \end{pmatrix}\in\C^{d_\xi\times d_\xi},
\end{equation*}
we have that $\sigma_P(\xi)\widehat{u}(\xi)=\widehat{f}(\xi)$ and, moreover,
\begin{align*}
    \|\widehat{u}(\xi)\|_{\HS}&=\| \tilde \sigma_P(\xi)^{-1}\widehat{f_1}(\xi)\|_{\HS}\\
    &\leq \| \tilde \sigma_P(\xi)^{-1}\|_{\operatorname{op}}\|\widehat{f}(\xi)\|_{\HS}\\
    &\leq C^{-1}\langle \xi \rangle^{k}\|\widehat{f}(\xi)\|_{\HS}.
\end{align*}
Due to the rapid decay of the Fourier coefficients of $f$, it follows that $u\in C^\infty(G)$. Moreover, the equality $\sigma_P(\xi)\widehat{u}(\xi)=\widehat{f}(\xi)$ for every $[\xi]\in\widehat{G}$ guarantees that $Pu=f$. This proves that $P$ is globally solvable.

Conversely, assume that inequality \eqref{ineq_diof_gs} does not hold. Then there exist sequences of distinct elements $[\xi_n]\in\widehat{G}$ and $v_{n}\in (\ker \sigma_P(\xi_n))^\perp\subset \C^{d_{\xi_n}}$, $\|v_n\|_2=1$, such that
\begin{equation*}
  0< \| \sigma_P(\xi_n)v_n\|_2<\langle \xi_n \rangle ^{-n},
\end{equation*}
for every $n\in\N$.

Consider $f\in C^\infty(G)$ given by the Fourier coefficients:
\begin{equation*}
    \widehat{f}(\xi)=\begin{cases}
        0&\text{ if }\xi\neq \xi_n,\ \forall n\in\N,\\
        \begin{pmatrix}
            \sigma_P(\xi_n)v_n&0&\cdots&0 
        \end{pmatrix}_{d_{\xi_n}\times d_{\xi_n}}&\text{ if }\xi= \xi_n,\ n\in\N.
    \end{cases}
\end{equation*}
Then, $f$ is well-defined, since 
\begin{equation*}
    \|\widehat{f}(\xi_n)\|_{\HS}=\|\sigma_P(\xi_n)v_n\|_2<\langle \xi_n \rangle ^{-n},
\end{equation*}
for every $n\in\N$ and $ \|\widehat{f}(\xi)\|_{\HS}=0$, for every other $[\xi]\in\widehat{G}$. Moreover, note that the columns of $\widehat{f}(\xi)$ are in $ \operatorname{ran}\sigma_P(\xi)$, for every $[\xi]\in\widehat{G}$, since 
\begin{equation*}
    \sigma_P(\xi_n)\begin{pmatrix}
            v_n&0&\cdots&0 
        \end{pmatrix}_{d_{\xi_n}\times d_{\xi_n}}=\widehat{f}(\xi_n),
\end{equation*}
for every $n\in\N$. Therefore by Proposition \ref{prop_kernel_annihi_range}, $f\in (\ker {}^tP)^0$. On the other hand, if $Pu=f$, then  the Fourier coefficients of $u$ must satisfy
\begin{equation*}
    \widehat{u}(\xi_n)=
        \begin{pmatrix}
            v_n&0&\cdots&0 
        \end{pmatrix}_{d_{\xi_n}\times d_{\xi_n}}+w_{\xi_n},
\end{equation*}
where the columns of $w_{\xi_n}\in \C^{d_\xi\times d_\xi}$ are in $\ker \sigma_P(\xi_n)$, and hence $\sigma_P(\xi_n)w_{\xi_n}=0$. Therefore, by orthogonality, we conclude that 
\begin{equation*}
    \|\widehat{u}(\xi_n)\|_{\HS}\geq \|v_n\|_2=1,
\end{equation*}
for every $n\in\N$. Therefore,  $u\in\mathcal{D}'(G)\backslash C^\infty(G)$ and, consequently, $P$ is not globally solvable.
\end{proof}

As an immediate consequence of the result above and Corollary \ref{coro_poincare_multipliers}, we obtain:
\begin{corollary}
    Let $G$ be a compact Lie group and $P:C^\infty(G)\to C^\infty(G)$ a left-invariant continuous linear operator. The following are equivalent:
    \begin{enumerate}
        \item $P$ is globally solvable (has closed range);
        \item There exist $c>0$ and $\delta\geq 1$ such that
        \begin{equation*}
        \|\nabla_{G} f\|_{L^2(G)}^{\delta-1}\| Pf\|_{L^2(G)}\geq c\|f\|_{L^2(G)}^\delta,
    \end{equation*}
    for every $f\in (\ker P|_{H^1(G)})^\perp$.
    \end{enumerate}
\end{corollary}

Proposition \ref{prop_globally_solvable}, together with the proof of Theorem \ref{theo_directional_sphere}, also implies in the following corollary.

\begin{corollary}\label{coro_GS_sphere}
    Every real left-invariant vector field on $\S$ is globally solvable.
\end{corollary}
\begin{proof}
    Let $Y=\alpha_1 D_1+\alpha_2 D_2+\alpha_3 D_3$ be an arbitrary real vector field on $\S$. Note that from the proof of Theorem \ref{theo_directional_sphere} we have that $\lambda_{\min}^{>0}[\sigma_Y(\ell)]\geq \frac{\|\alpha\|}{2}$, for every $\ell\in\frac{1}{2}\N_0$. Therefore by Proposition \ref{prop_globally_solvable}, $Y$ is globally solvable.
\end{proof}

\section{Directional Poincaré inequality for tube-type vector fields}\label{sect_tube}

In view of the relationship between the directional Poincaré inequality and global solvability proved in the previous section, together with the fact that a (real) tube-type vector field on $\T^1\times G$ is globally solvable if and only the corresponding normalized vector field is globally solvable, in this section we explore a class of tube-type vector fields on $\T^1\times G$, and show that an equivalence between the existence of a directional Poincaré inequality and the global solvability of such vector fields holds as in the left-invariant case.

\begin{obs}\label{remark_normal_ran_ker}
For any Fourier multiplier $P$ which is normal with respect to the $L^2$ inner product, the fact that 
    \begin{equation*}
        (\ker \sigma_P(\xi))^\perp=\operatorname{ran}\sigma_P(\xi)
    \end{equation*}
    (as mentioned in the proof of Proposition \ref{prop_composed_ineq_general}) together with  Propositions \ref{prop_EP_equiv_Fourier} and  \ref{prop_kernel_annihi_range}, implies that
    \begin{equation*}
        (\ker {}^tP)^0=(\ker P|_{H^1(G)})^\perp\cap C^\infty(G).
    \end{equation*}
    In particular, since left-invariant vector fields are antisymmetric, therefore normal, with respect to the $L^2$ inner product, the equality above also holds for left-invariant vector fields on $G$. With that in mind, an  equivalent characterization of the set $(\ker Y|_{H^1(G)})^\perp$  can be obtained as follows: we have that $(\ker Y|_{H^1(G)})^\perp=\overline{(\ker {}^tY)^0}$, where the closure is taken with respect to the $H^1(G)$ norm on $C^\infty(G)$. It is worth mentioning that this holds more generally: for any real vector field $Y$ on $G$, using that ${}^tY=-Y$ on $\mathcal{D}'(G)$ and that $Y$ is real, we can also conclude that  $(\ker {}^tY)^0=(\ker Y|_{H^1(G)})^\perp\cap C^\infty(G)$, and so the previous characterization also holds in this case. 
\end{obs}

Before stating and proving the main result of this section, we fix the following notation: let $X$ be a left-invariant vector field on a compact Lie group $G$. Since $X$ antisymmetric with respect to the $L^2$ inner product on $G$, 
 one can show (see \cite[Remark 10.4.20]{RT2010_book})  its symbol $\sigma_X(\xi)$  is anti-hermitian for every $[\xi]\in\widehat{G}$. Therefore its eigenvalues are purely imaginary and we denote them by $i\mu_r(\xi)\in i\R$, for  $1\leq r\leq d_\xi$ and every $[\xi]\in\widehat{G}$. Moreover, we can always choose a representative $\xi\in [\xi]$ such that $\sigma_X(\xi)$ is diagonal.

\begin{theorem}\label{theo_poincare_tube}
    Let $a(t)\in C^\infty(\T^1,\R)$ and $X$ a left-invariant vector field on a compact Lie group $G$. Consider the real vector field
\begin{equation}\label{eq_tube_vector}
    Y= \partial_t+a(t)X
\end{equation}
 on $\T^1\times G$. Given $\delta\geq 1$, there exists $c>0$ such that
\begin{equation}\label{ineq_poincare_variable_coef}
    \|\nabla_{\T^1\times G}f\|_{L^2(\T^1\times G)}^{\delta-1}\| Yf\|_{L^2(\T^1\times G)}\geq c\|f\|_{L^2(\T^1\times G)}^\delta,
\end{equation}
for every $f\in (\ker Y|_{H^1(\T^1\times G)})^\perp$, 
if and only if there exists $C>0$ such that
\begin{equation}\label{eq_diof_tube}
    \min_{1\leq r\leq d_\xi}| k+a_0\mu_r(\xi) |\geq C(|k|+\langle\xi\rangle)^{-(\delta-1)},
\end{equation}
for every $(k,[\xi])\in\Z\times G$ satisfying $ \min_{1\leq r\leq d_\xi}| k+a_0\mu_r(\xi) |\neq 0$, where
\begin{equation*}
    a_0=\frac{1}{2\pi}\int_{0}^{2\pi} a(t)dt.
\end{equation*}
\end{theorem}

\begin{proof}
For each $[\xi]\in\widehat{G}$, choose a representative $\xi$ of $[\xi]$ such that $\sigma_X(\xi)$ is diagonal. Let $\Psi:C^\infty(\T^1\times G)\to \C^\infty(\T^1\times G)$ be given by
\begin{equation*}
    \Psi f(t,x)=\sum_{[\xi]\in\widehat{G}}d_\xi \sum_{r,s=1}^{d_\xi} e^{-i\mu_r(\xi)(\int_0^t a(\tau)d\tau-a_0t)}\widehat{f}(t,\xi)_{rs}\xi_{sr}(x),
\end{equation*}
for $f\in C^\infty(\T^1\times G)$.
 It is easy to check that  $\|\Psi f\|_{L^2}=\|\Psi^{-1} f\|_{L^2}=\|f\|_{L^2}$, $\Psi$ is bijective and $\Psi^{-1}Y\Psi =Y_0$, where
 \begin{equation*}
    \Psi^{-1} f(t,x)=\sum_{[\xi]\in\widehat{G}}d_\xi \sum_{r,s=1}^{d_\xi} e^{i\mu_r(\xi)(\int_0^t a(\tau)d\tau-a_0t)}\widehat{f}(t,\xi)_{rs}\xi_{sr}(x);
\end{equation*}
 see \cite{KWR_product} for more details. 

 Next, note that for $f\in C^\infty(\T^1\times G)$, we have that 
\begin{align*}
    \|\nabla &(\Psi^{-1} f)\|_{L^2(\T^1\times G)}^2=\sum_{[\xi]\in\widehat{G}}d_\xi\sum_{r,s=1}^{d_\xi}\int_0^{2\pi}|\partial_t[\widehat{\Psi^{-1}f}(t,\xi)_{rs}|^2+|\mu_r(\xi)\widehat{\Psi^{-1}f}(t,\xi)_{rs}|^2dt\\
    &=\sum_{[\xi]\in\widehat{G}}d_\xi\sum_{r,s=1}^{d_\xi}\int_0^{2\pi} |\partial_t\widehat{f}(t,\xi)_{rs}|^2+|\mu_{r}(\xi)(a(t)-a_0)\widehat{f}(t,\xi)_{rs}|^2+|\mu_{r}(\xi)\widehat{f}(t,\xi)_{rs}|^2dt\\
    &\leq \sum_{[\xi]\in\widehat{G}}d_\xi\sum_{r,s=1}^{d_\xi}\int_0^{2\pi}((a(t)-a_0)^2+1)(|\partial_t\widehat{f}(t,\xi)_{rs}|^2+|\mu_{r}(\xi)\widehat{f}(t,\xi)_{rs}|^2)dt\\
    &\leq \Big(\max_{0\leq s\leq 2\pi}(a(s)-a_0)^2+1\Big)\sum_{[\xi]\in\widehat{G}}d_\xi\sum_{r,s=1}^{d_\xi}\int_0^{2\pi}|\partial_t\widehat{f}(t,\xi)_{rs}|^2+|\widehat{Xf}(t,\xi)_{rs}|^2dt\\
    &=(\operatorname{var}(a)+1)\|\nabla f\|_{L^2(\T^1\times G)}^2,
\end{align*}
where $\operatorname{var}(a)\vcentcolon=\max_{0\leq s\leq 2\pi}(a(s)-a_0)^2$. Moreover, one can show (see  \cite[proof of Proposition 5.1.2]{Kow_thesis}) that the restriction $\Psi^{-1}:(\ker {}^tY)^0\to (\ker {}^tY_0)^0$ is well defined and bijective. 

Note that under the identifications $\widehat{\T^1\times G}\sim \widehat{\T^1}\times \widehat{G}$ and  $\widehat{\T^1}\sim \Z$ we have that 
\begin{equation}\label{eq_symbol_y0}
    \sigma_{Y_0}(k,\xi)=ik\operatorname{Id}_{d_\xi}+a_0\sigma_X(\xi),
\end{equation}
for every $(k,\xi)\in\Z\times G$ (again see \cite{KWR_product} for more details).
Therefore, if \eqref{eq_diof_tube} holds, by Theorem \ref{theo_directional_poincare_general}  the vector field $Y_0$ satisfies the directional Poincaré inequality with exponent $\delta$ and some constant $c_0>0$, so by Remark \ref{remark_normal_ran_ker} for every $f\in (\ker {}^tY)^0$ we have that
\begin{align*}
    \|\nabla f\|_{L^2}^{\delta-1}\|Yf\|_{L^2}&\geq \frac{1}{(\operatorname{var}(a)+1)^{\frac{\delta-1}{2}}}\|\nabla (\Psi^{-1}f)\|_{L^2}^{\delta-1}\|\Psi^{-1}Y(\Psi \Psi ^{-1}f)\|_{L^2}\\
    &=\frac{1}{(\operatorname{var}(a)+1)^{\frac{\delta-1}{2}}}\|\nabla (\Psi^{-1}f)\|_{L^2}^{\delta-1}\|Y_0\Psi^{-1}f\|_{L^2}\\
    &\geq  \frac{c_0}{(\operatorname{var}(a)+1)^{\frac{\delta-1}{2}}}\|\Psi^{-1}f\|_{L^2}^\delta= \frac{c_0}{(\operatorname{var}(a)+1)^{\frac{\delta-1}{2}}}\|f\|_{L^2}^\delta.
\end{align*}
As also mentioned in Remark \ref{remark_normal_ran_ker}, $ (\ker {}^tY)^0= (\ker Y|_{H^1(\T^1\times G)})^\perp\cap C^\infty(\T^1\times G)$,
hence this space is dense in $(\ker Y|_{H^1(\T^1\times G)})^\perp$ with respect to the $H^1$ norm, and we conclude that inequality \eqref{ineq_poincare_variable_coef} holds with $c={c_0}/{(\operatorname{var}(a)+1)^{\frac{\delta-1}{2}}}$ for every $f\in (\ker Y|_{H^1(\T^1\times G)})^\perp$.

Conversely, assume that inequality \eqref{ineq_poincare_variable_coef} holds for some $c>0$, and every $f\in (\ker Y|_{H^1(\T^1\times G)})^\perp$. Note that by the same arguments used previously, we also have that
\begin{equation*}
    \|\nabla_{\T^1\times G}(\Psi f)\|_{L^2(\T^1\times G)}\leq (\operatorname{var}(a)+1)^{1/2}\|\nabla f\|_{L^2(\T^1\times G)}.
\end{equation*}
Hence, for every $f\in (\ker {}^tY_0)^0$, we have
\begin{align*}
    \|\nabla f\|_{L^2}^{\delta-1}\|Y_0f\|_{L^2}&\geq \frac{1}{(\operatorname{var}(a)+1)^{\frac{\delta-1}{2}}}\|\nabla (\Psi f)\|_{L^2}^{\delta-1}\|\Psi Y_0(\Psi^{-1} \Psi f)\|_{L^2}\\
    &=\frac{1}{(\operatorname{var}(a)+1)^{\frac{\delta-1}{2}}}\|\nabla (\Psi f)\|_{L^2}^{\delta-1}\|Y\Psi f\|_{L^2}\\
    &\geq  \frac{c}{(\operatorname{var}(a)+1)^{\frac{\delta-1}{2}}}\|\Psi f\|_{L^2}^\delta= \frac{c}{(\operatorname{var}(a)+1)^{\frac{\delta-1}{2}}}\|f\|_{L^2}^\delta.
\end{align*}
Once again by the fact that $ (\ker {}^tY_0)^0\subset (\ker Y_0|_{H^1(\T^1\times G)})^\perp$ is dense  with respect to the $H^1$ norm (see Remark \ref{remark_normal_ran_ker}), we conclude that the inequality above holds for every $f\in (\ker Y_0|_{H^1(\T^1\times G)})^\perp$.
 By Theorem \ref{theo_directional_poincare_general} and  the identifications $\widehat{\T^1\times G}\sim \widehat{\T^1}\times \widehat{G}$,  $\widehat{\T^1}\sim \Z$, as well as the formula for the symbol \eqref{eq_symbol_y0}, 
we conclude that \eqref{eq_diof_tube} must hold.
\end{proof}

\begin{corollary}\label{coro_tube_1}
    Let $Y$ be the vector field on $\T^1\times G$ given by \eqref{eq_tube_vector} and fix $\delta\geq 1$. Inequality \eqref{ineq_poincare_variable_coef} holds for some $c>0$ if and only if there exists $c_0>0$ such that 
    \begin{equation}\label{eq_directional_Y0}
        \|\nabla f\|_{L^2(\T^1\times G)}^{\delta-1}\| Y_0f\|_{L^2(\T^1\times G)}\geq c_0\|f\|_{L^2(\T^1\times G)}^\delta,\quad \forall f\in (\ker Y_0|_{H^1(\T^1\times G)})^\perp,
    \end{equation}
  where
    \begin{equation*}
        Y_0=\partial_t+a_0X,
    \end{equation*}
    with $a_0=\frac{1}{2\pi}\int_{0}^{2\pi}a(t)dt$,
     and we can take 
        \begin{equation*}
        c_0=\frac{c}{(\operatorname{var}(a)+1)^{\frac{\delta-1}{2}}}\quad\text{or}\quad c=\frac{c_0}{(\operatorname{var}(a)+1)^{\frac{\delta-1}{2}}},
    \end{equation*}
    if the inequalities hold.\\
    Moreover, the following statements are equivalent:
    \begin{enumerate}
        \item  Inequality \eqref{ineq_poincare_variable_coef} holds for some $\delta\geq 1$ and $c>0$;
        \item Inequality \eqref{eq_directional_Y0} holds for some $\delta\geq 1$ and $c_0>0$;
        \item $Y$ is globally solvable;
        \item $Y_0$ is globally solvable.
    \end{enumerate}
\end{corollary}

\begin{corollary}\label{coro_tube_t2}
     Let $a\in C^\infty(\T^1)$ be real-valued and consider the vector field on $\T^2$ given by
     \begin{equation*}
         Y=\partial_t+a(t)\partial_x.
     \end{equation*}
     There exist $\delta\geq 2$ and  $c_a>0$ such that
          \begin{equation}\label{ineq_coro_tube_t2}
        \|\nabla_{\T^2} f\|_{L^2(\T^2)}^{\delta-1}\| Y f \|_{L^2(\T^2)}\geq c_a\|f\|_{L^2(\T^2)}^\delta,
    \end{equation}
    for every $f\in H^1(\T^2)$ with mean-value zero, if and only if 
     $a_0=\frac{1}{2\pi}\int_0^{2\pi} a(t)dt$ is an irrational non-Liouville number. Moreover: in this case we can take $\delta$ greater then the irrationality measure of $a_0$, or equal to $2$ is $a_0$ is algebraic of degree $2$. 
\end{corollary}
\begin{proof}
Suppose first that $a_0$ is an irrational non-Liouville number. From Corollary \ref{coro_irrational} we conclude that a directional Poincaré inequality holds for $Y_0$ for every $f\in H^1(\T^2)$ with mean-value zero. Now let $f\in (\ker {}^tY)^0$. As in the proof of Theorem \ref{theo_poincare_tube}, we have that $\Psi^{-1}f\in (\ker {}^tY_0)^0$, which coincides with $(\ker Y_0|_{H^1(\T^2)})^\perp\cap C^\infty(\T^2)$, by Remark \ref{remark_normal_ran_ker}. As in the proof of Corollary \ref{coro_irrational}, this implies that $\widehat{\Psi^{-1}f}(0,0)=0$, that is: 
\begin{align}\label{eq_mean_value}
    0&=\widehat{\Psi^{-1}f}(0,0)=\frac{1}{(2\pi)}\int_0^{2\pi}\widehat{{\Psi^{-1}f}}(t,0)dt\notag\\
    &=\frac{1}{(2\pi)}\int_0^{2\pi}\widehat{{f}}(t,0)dt=\frac{1}{(2\pi)^2}\int_0^{2\pi}\int_0^{2\pi}f(t,x)dtdx,
\end{align}
so every $f\in (\ker {}^tY)^0$ has mean-value zero. From the density of $(\ker {}^tY)^0\subset (\ker Y|_{H^1(\T^2)})^\perp$, we conclude that this property also holds for any $f\in (\ker Y|_{H^1(\T^2)})^\perp$. Inequality \eqref{ineq_coro_tube_t2} and then follows Corollary \ref{coro_tube_1} ({\it (2)}$\implies ${\it (1)}). The other claims about the exponent also follow from Corollary \ref{coro_irrational}.

Next suppose that there exist $\delta\geq 2$ and $c_a>0$ such that inequality \eqref{ineq_coro_tube_t2} holds for every $f\in H^1(G)$ with mean value zero. This implies that the subspace $\{f\in H^1(G): f \text{ has mean value zero}\}$ is contained in $(\ker Y|_{H^1(\T^2)})^\perp$, by Remark \ref{remark_perp}. In particular, by Remark \ref{remark_normal_ran_ker}, we have that $M_0\vcentcolon=\{f\in C^\infty(\T^2): f \text{ has mean value zero}\}\subset (\ker {}^tY)^0=(\ker Y|_{H^1(\T^2)})^\perp\cap C^\infty(\T^2)$. 
 Following the notation and arguments in the proof of Theorem \ref{theo_poincare_tube}, this implies that there exists $c_0>0$ such that for any $f\in \Psi^{-1}M_0$, $f\in  (\ker {}^tY_0)^0$, inequality 
\begin{equation}\label{ineq_Y0_proof}
     \|\nabla_{\T^2} f\|_{L^2(\T^2)}^{\delta-1}\| Y_0 f \|_{L^2(\T^2)}\geq c_0\|f\|_{L^2(\T^2)}^\delta,
\end{equation}
holds, where
    \begin{equation*}
        Y_0=\partial_t+a_0\partial_x.
    \end{equation*}
    But by identities \eqref{eq_mean_value}, $\Psi^{-1}M_0=M_0$, so Remark \ref{remark_perp} implies that $M_0 $ is a subspace of $(\ker Y_0|_{H^1(G)})^\perp$, and we conclude that $a_0$ must be irrational. But then $M_0=(\ker {}^tY_0)^0$, so by density we conclude that \eqref{ineq_Y0_proof} holds for every $f\in (\ker Y_0|_{H^1(\T^2)})^\perp$. It then follows from Corollary \ref{coro_t2} that $a_0$ must be a non-Liouville number.
\end{proof}

\section{Acknowledgements}
The authors thank Professor Noam D. Elkies who provided the ideas for the proof of Lemma \ref{lemma_tridiagonal_eig} (see ``Independence of parameter for eigenvalues of periodic family of
tridiagonal matrices", MathOverflow; URL \url{https://mathoverflow.net/q/498252}).

The first author was supported in part by FAPESP (grants 2024/08416-6 and 2024/12753-8) and CNPq (grant 313581/2021-5). The second author was supported by FAPESP (grant 2025/08151-5).
\bibliographystyle{elsarticle-num} 
  \bibliography{references.bib}
\end{document}